\RequirePackage{fix-cm}
\documentclass[smallcondensed]{svjour3}     
\smartqed  
\usepackage{graphicx}

\usepackage{latexsym,amsmath,amssymb}
\usepackage{hyperref}
\usepackage{mathrsfs}

\usepackage [latin1]{inputenc}

\usepackage{graphicx}
\usepackage{epstopdf}
\usepackage{bm}

\DeclareGraphicsExtensions{.pdf,.jpg,.png,.tif}
\usepackage{dcolumn}
\usepackage{bm}
\usepackage{hyperref}
\usepackage{mathrsfs}
\usepackage{epsfig}
\usepackage{multirow}
\usepackage{color}

\usepackage{amsfonts}

\DeclareRobustCommand{\vect}[1]{\bm{#1}}
\begin{document}

\title{Adaptive finite element approximation for steady-state Poisson-Nernst-Planck equations
}

\titlerunning{Adaptive finite element method for PNP equations}        

\author{Tingting Hao        \and
        Manman Ma \and
        Xuejun Xu 
}


\institute{Tingting Hao \at
              School of Mathematical Sciences, Tongji University, Shanghai 200092, China \\
              \email{1710386@tongji.edu.cn}           
           \and
           Manman Ma \at
              School of Mathematical Sciences, Tongji University, Shanghai 200092, China\\
              \email{mamm@tongji.edu.cn}
            \and
            Xuejun Xu \at
            School of Mathematical Sciences, Tongji University, Shanghai 200092, China, and LSEC, Institute of Computational Mathematics and Scientific/Engineering Computing, Academy of Mathematics and System Sciences, Chinese Academy of Sciences, Beijing 100190, China\\
            \email{xxj@lsec.cc.ac.cn}            
}

\date{Received: date / Accepted: date}

\maketitle

\begin{abstract}
In this paper, we develop an adaptive finite element method for the nonlinear steady-state Poisson-Nernst-Planck equations, where the spatial adaptivity for geometrical singularities and boundary layer effects are mainly considered. As a key contribution, the steady-state Poisson-Nernst-Planck equations are studied systematically and rigorous analysis for a residual-based a posteriori error {estimate} of the nonlinear system is presented. With the help of Schauder fixed point theorem, we show the solution existence and uniqueness of the linearized system derived by taking $G-$derivatives of the nonlinear system, followed by the proof of the relationship between the error of solution and the a posteriori error estimator $\eta$.
Numerical experiments are given to validate the efficiency of the a posteriori error estimator and demonstrate the expected rate of convergence. In the further tests, adaptive mesh refinements for geometrical singularities and boundary layer effects are successfully observed.
\end{abstract}

\section{Introduction}
In this work, we consider the steady-state Poisson-Nernst-Planck (PNP) equations
\begin{equation}\label{model}
\begin{cases}
\begin{array}{lll}
&-\nabla\cdot D^i(\nabla c^i+\beta z^iec^i \nabla\phi )=F^i,\ i=1,...,K, \\
&-\nabla\cdot(\varepsilon_r\varepsilon_0\nabla \phi) - \sum\limits_{i=1}^Kz^iec^i = F^f,
\end{array}
\end{cases}
\end{equation}
which describe the nonlinear coupling of the electric potential $\phi$ and the ionic concentration $c^i$ of the $i$-th species. The first equation is called the Nernst-Planck equation for species $i$ and the second one is the named the Poisson's equation.
Here, $D^i$ and $z^i$ are the diffusivity and valence, respectively. $\beta=1/(k_BT_{abs})$, with $k_B$ the Boltzmann constant and $T_{abs}$ the absolute temperature. $e$ is the elementary charge, $\varepsilon_r$ and $\varepsilon_0$ are the relative and vacuum dielectric permittivities, $F^i$ denotes the reaction source term, and $F^f$ is generally due to the fixed charges.

Despite the fact that the PNP system has been well known and widely studied for over a century because of its physical applications such as semiconductor studies~\cite{Brezzi2005,SGajewski,SMauri} and electrodiffusion problems \cite{electrodiffusion1,electrodiffusion2,electrodiffusion3}, it is attracting more attentions as being used to describe the dynamics of ion transport in biological membrane channels \cite{FELu,wei2012SIAM,xu2014MBMB,tu2015CPC}. Recently, modified models beyond
mean field theory such as incorporating ionic steric effects, correlation effects and inhomogeneous dielectric effects have been extensively studied with great interests~\cite{kilic2007bPRE,lu2011BJ,horng2012JPCB,xu2014PRE,Jiang2014JPCM}. For early theoretical analysis, Jerome  \cite {Jerome1} and Hayeck \cite{Hayeck1990EXISTENCE} proved  the existence of solutions  for steady-state PNP according to Schauder fixed point theorem, while Mock~\cite{mock2}, Brezzi~\cite{Brezzi}, and Gajewski~\cite{Gajewski} gave out the conclusion of uniqueness under some local constraints. Due to the strong coupling and high nonlinearity, the analytical or asymptotic solutions of PNP equations have  been only studied for simple 1-D~\cite{golovnev2011JCP} or single-ion-species cases~\cite{schonke2012JPA}.  More analytical results for steady-state PNP equations are referred to Refs.~\cite{barcilon1997SIAM,AGolovnev,SPLiu,singer2008EJAM}
where only 1-D cases are considered. As a result, the steady or unsteady PNP equations are in general solved numerically with regular computational domains.

A variety of numerical methods have been used extensively for solving PNP equations, which can be broadly categorized as using finite difference method \cite{FDBolintineanu,FDCardenas,FDZheng,liu2017free,Ding19,hu2020fully}, finite element method \cite{FEsungao,FEsunhem,FELu,FEsongzhang,FEsongshen,BFZhou,FEjerome,metti2016energetically}, finite volume method \cite{FVMathur,FVSrinivasan}, the boundary element method \cite{BFZhou},  and the spectrum method \cite{SPHollerbach}. Among those numerical approaches, finite element methods have been known to perform well for more irregular geometries and complicated boundaries. With the increasing computing capabilities and numerical improvements, the PNP model has been applied recently in simulating large systems such as practical biophysical system, nevertheless, the computational costs of classical methods with uniform meshes are still expensive. Additionally, the boundary layer effect due to thin Debye layer, the singularity of Dirac charge distribution sources, and geometrical singularities cause significant difficulties in obtaining required numerical accuracy when standard methods are used.
Although singular perturbation treatments have been applied to the PNP system for thin Debye layer limit~\cite{SPLiu}, very few efficient numerical analysis aiming at overcoming those singularities has been proposed to the best of our knowledge.
As a special equilibrium state of PNP equation, the nonlinear Poisson-Boltzmann (PB) equation with Dirac charge distribution has been systematically discussed and adaptive finite element approximations have been developed in previous works~\cite{holst2000jcp,PBchenlong}. Later on, a parallel adaptive finite element algorithm of the steady-state PNP has been developed with using an a posteriori error estimator which is similar to that in PB \cite{xie2013MBMB}.  However, PB has  reduced the strongly coupled system into a single equation and simplified the nonlinear difficulties of PNP system itself. Besides, the above mentioned singularities and boundary layer effects require more general adaptive numerical analysis for PNP equations.

The analysis of adaptive finite element methods has made important progress in understanding the basic principles in recent years.
Typical adaptive algorithms provide an automatic feedback routine with successive loops of the structure: {\it SOLVE $\rightarrow$ ESTIMATE $\rightarrow$ MARK $\rightarrow$ REFINE},
of which `{\it ESTIMATE} ' is to find out where correction is needed through the a posteriori error estimator $\eta$, so as to prepare for grid-marking and mesh-refinement. The adaptive finite element method is able to drive local refinement by a posteriori error estimation and achieve the mesh adaptation, which can well resolve the geometrical singularities, boundary layer effects, and so on. Meanwhile,  the a posteriori error estimates determine the error control that depends on the numerical solutions only. {A} posteriori error estimation varies for different equations \cite{Carstensen,VerfurthR,Eriksson} and multiple a posteriori error estimates can exist even for the same problem  \cite{Oden}. In this work, the residual-based a posteriori error estimates are mainly considered for analyzing the PNP equations, however, the nonlinearity and strong coupling of the system lead to great difficulties in obtaining the estimates as following the classical {a} posteriori error estimation of the residual type \cite{Verfuhrtgeneral} where the a posteriori error estimates can be achieved by considering the original equation system directly.
Fortunately, a general solution for the a posteriori error estimates of nonlinear elliptic equations was developed by {Verf\"{u}rth}~\cite{Verfuhrt}  decades ago and was further discussed by Ainsworth and Oden in Ref.~\cite{Oden}.
{Besides, there are many studies on the a posterior estimation of linear singular perturbation problems with considering the boundary layer effects}
\cite{ainsworth1999reliable,demlow2016maximum-norm,cheddadi2009guaranteed,ainsworth2011fully,zhang2015guaranteed}. However, to solve cases with nonlinear and strongly coupling problems is challenging. The current work proposes the a posteriori error estimate of PNP equations and proves both its reliability and efficiency.  {Furthermore, we demonstrate the robustness of the reliability which is independent of the mesh size and the parameter denoting the thin Debye layer.}

In {Verf\"{u}rth}'s theory, the resulted linear equation system  after taking $G-$ derivatives of the original nonlinear system has to satisfy some preconditions, the most stringent of which is the solution uniqueness. Noteworthily, we present certain constraints under which the solution uniqueness is proved and that was absent in the theoretical analysis in \cite{PBchenlong}. Since the canonical method for proving the solution existence of linear partial differential equations~\cite{Gilbarg} is impractical here due to the special structure of the PNP system,
we instead construct operators similarly as developed in Ref.~\cite {Jerome1} to decouple the nonlinear equation system, and give the  solution existence of the system by utilizing the Schauder fixed point theorem~\cite{Zeidler} and Poisson equation properties~\cite{Gilbarg}.

The paper is organized as follows. We begin in Section \ref{sec2} by defining the PNP system, its variational form, and the finite element spaces to be used. In Section \ref{sec:posterr}, we prove the solution existence and uniqueness of linearized equations with details presented in \ref{app:existanduniq}, and develop the a posteriori error estimate with demonstrating its reliability and efficiency.
In Section \ref{numerics}, numerical experiments are reported to show the accuracy of the adaptive method and agreements to theorems of the a posteriori error estimation. Conclusion remarks are finally made in Section \ref{conclusions}.

\section{Preliminary} \label{sec2}
Let $\Omega\subset R^2$ be a bounded Lipschitz domain. We use standard notations for Sobolev spaces 
$W^{1,2}(\Omega)=H^1(\Omega)$,
and the associated norms and seminorms.
Here, $H_0^1 (\Omega) =\{v\in H^1(\Omega): v|_{\partial \Omega} =0\}$ and $H^{-1}(\Omega)=(H_0^1(\Omega))^*$ which means the dual space of $H_0^1(\Omega)$.
In all proofs,
$a\lesssim b$ means $a\leq C b$ where $C$ is a constant.

\subsection{PNP system and its variational form}
For simplicity, the 1:1 symmetric system is considered in this work, with $c^1=p$ and $c^2=n$ denoting the positive and negative ionic concentration, respectively. The general steady-state PNP is then reduced to the dimensionless form as follows:
\begin{equation} \label{model2}
\begin{cases}
\begin{array}{l}
-\Delta p-\nabla\cdot(p\nabla\psi)=f_1, \\
-\Delta n+\nabla\cdot(n\nabla\psi)=f_2, \\
-\epsilon\Delta \psi=p-n+f_3,
\end{array}
\end{cases}
\end{equation}
here the parameter $\epsilon=2(\ell_D/l_0)^2$ with $\ell_D$ the Debye length and $l_0$ the characteristic length of the computational domain. {In this work, we take $\epsilon\in (0,1]$ where small $\epsilon$ values correspond to cases with boundary layer effects, or namely thin Debye layer effects.}
	

Alternatively, the system \eqref{model2} can be written as a general system $\bm{F}(\bm{u}) = \bm{0}$ with $\bm{u}(\bm{x}) = (p(\bm{x}),n(\bm{x}),\psi(\bm{x}))$, and is satisfied by {$p(\bm{x}),n(\bm{x})$ and $\psi(\bm{x}) \in H^1({\Omega}) $ on the two dimensional domain considered in current work, i.e., $\bm{x}=(x,y)\in \Omega$. Let $\bm{f}(\bm{x})=(f_1,f_2,f_3)\in (L^{2}(\Omega))^3$} and the system satisfy Dirichlet  boundary condition. Hence, the weak formulation of system \eqref{model2} is,
\begin{equation}
\int_{\Omega}
(v_1, v_2, v_3)
\left(                 %
\begin{array}{c}
 -\Delta p-\nabla\cdot(p\nabla\psi)\\
 -\Delta n+\nabla\cdot(n\nabla\psi)\\
 -\epsilon\Delta \psi-p+n \\
 \end{array}
\right) d\bm{x}
=
\int_{\Omega}
(v_1, v_2, v_3)
\left(                 %
\begin{array}{c}
 f_1\\
f_2\\
f_3 \\
 \end{array}
\right) d\bm{x},
\end{equation}
then, ${\bm{u}}=(p,n,\psi)$ is the weak solution of \eqref{model2} if and only if $\forall {\bm{v}}= (v_1, v_2, v_3)\in (H_0^1(\Omega))^3$, ${\bm{u}}$ satisfies,
\begin{align}
<{\bm{F}}({\bm{u}}),{\bm{v}}>= &\int_{\Omega}(\nabla p\cdot\nabla v_1+p\nabla\psi\cdot\nabla v_1+\nabla n \cdot\nabla v_2-n\nabla\psi\cdot \nabla v_2\nonumber \\
&\quad+{\epsilon}\nabla\psi\cdot \nabla v_3-pv_3+nv_3-f_1v_1-f_2v_2-f_3v_3)d\bm{x} \nonumber\\
= & 0.	\label{vmodel2}
\end{align}
Note that ${\bm F}$ is a nonlinear operator here from ${(H^1({\Omega}) )^3}$ to $(H^{-1}(\Omega))^3.$

{Next, we} propose the linear problem by introducing the $G$-derivative ($G\hat{a}teaux$ derivative~\cite{Zeidler}) operator ${\bm {DF}}(\cdot)$ at ${\bm{u}}$ as follows,
\begin{equation}\label{definition}
<{\bm {DF}}({\bm{u}}){\bm{\phi}},{\bm{w}}> = \lim\limits_{\widetilde{\varepsilon}\rightarrow 0}\frac{<{\bm F}({\bm{u}}+\widetilde{\varepsilon}{\bm{\phi}})-{\bm F}({\bm{u}}),{\bm{w}}>}{\widetilde{\varepsilon}},\ \forall {\bm{\phi}}, \bm{w}\in (H_0^1(\Omega))^3, 
\end{equation}
which leads to the following linear problem: to find ${\bm{\phi}}=(\phi_1,\phi_2,\phi_3)\in (H_0^1(\Omega))^3$, $\forall {\bm{R}}\in (L^2(\Omega))^3$,  such that,
\begin{equation}\label{modelDF}
<{\bm {DF}}({\bm{u}}){\bm{\phi}},{\bm{v}}> = <{\bm{R}},{\bm{v}}>,\quad  \forall {\bm{v}}\in  (H_0^1(\Omega))^3,
\end{equation}
where ${\bm{u}}$ satisfies Eq. \eqref{vmodel2}.

\subsection{Finite element discretization}
Let the quasi-uniform mesh of $\Omega$ be $\mathcal{T}_h$ with mesh size $0<h<1$ \cite{Brenner} and the corresponding finite element space be $V_h(\Omega)=\{ v\in H_0^1(\Omega):v|_{\tau}\in \mathcal{P}_1(\tau),\forall \tau\in \mathcal{T}_h  \}$. We  define ${\bm{u}_h}:=(p_h,n_h,\psi_h)$ with
\begin{align*}
&p_h\in P_h= \{v\in H^1(\Omega):v|_{T}\in \mathcal{P}_1(T),\forall T\in \mathcal{T}_h,v|_{\partial \Omega}=\bar{p }\}, \\
&n_h\in N_h=\{v\in H^1(\Omega):v|_{T}\in \mathcal{P}_1(T),\forall T\in \mathcal{T}_h,v|_{\partial \Omega}=\bar{n }\}, \\
&\psi_h\in PSI_h=\{v\in H^1(\Omega):v|_{T}\in \mathcal{P}_1(T),\forall T\in \mathcal{T}_h,v|_{\partial \Omega}=\bar{\psi }\},
\end{align*}
here, $\bar{p}, \bar{n}$ and $\bar{\psi} \in H^1(\Omega)$ satisfy, $\Gamma_0 \bar{p} = {p}, \Gamma_0 \bar{n} = {n}$ and $\Gamma_0 \bar{\psi} = \psi$, on $\partial \Omega $ with $\Gamma_0$ the trace operator. The finite element approximation of \eqref{vmodel2} is {to find} ${\bm{u}_h} \in P_h \times N_h \times PSI_h$ satisfying,
\begin{equation}\label{weakequation-1}
<{\bm F}({\bm{u}_h}),{\bm{v}_h}>=0, \ \forall {\bm{v}_h}=(v_{h,1},v_{h,2},v_{h,3})\in {(V_h)^3},
\end{equation}
which based on \eqref{vmodel2} gives
\begin{align}
	<{\bm F}({\bm{u}_h}),{\bm{v}_h}>
=&\int_{\Omega}(\nabla p_h\cdot\nabla v_{h,1}+p_h\nabla\psi_h\cdot\nabla v_{h,1}+\nabla n_h \cdot\nabla v_{h,2}-n_h\nabla\psi_h\cdot \nabla v_{h,2}  \nonumber\\
+\epsilon\nabla &\psi_h\cdot \nabla v_{h,3}-p_hv_{h,3}+n_hv_{h,3}-f_1v_{h,1}-f_2v_{h,2}-f_3v_{h,3})d{\bm x}. 
\end{align}

\section{The construction of {an} a posteriori error estimator} \label{sec:posterr}
{The error in this work is measured in an $\epsilon$-dependent norm $||\cdot||_{\epsilon,\Omega}$, i.e.,
$\forall \phi\in H^1(\Omega)$,
\begin{equation*}
||\phi||_{\epsilon,\Omega}=\left(\|\phi\|_{L^2(\Omega)}^2+\epsilon|\phi|_{1,\Omega}^2\right)^{{1}/{2}},
\end{equation*}
and
\begin{equation*}
||{\bm u}||_{\epsilon,\Omega}=\left(\|p\|_{\epsilon,\Omega}^2+\|n\|_{\epsilon,\Omega}^2+\|\psi\|_{\epsilon,\Omega}^2\right)^{{1}/{2}}.
\end{equation*}
The similar treatment is referred to \cite{verfurth1998robust}.}
We further define 
\begin{equation*}
{||f||_{-\epsilon,\Omega}=\sup\limits_{v\in H_0^1(\Omega)}\frac{|<f,v>|}{\|v\|_{\epsilon,\Omega}}},\quad  {||\bm{F}||_{-\epsilon,\Omega}=\sup\limits_{\bm{v}\in (H_0^1(\Omega))^3}\frac{|<\bm{F},\bm{v}>|}{\|\bm{v}\|_{\epsilon,\Omega}}},
\end{equation*}
$\forall f\in H^{-1}(\Omega)$ and $\forall \bm{F}\in ((H_0^{1}(\Omega))^3)^*$ where $((H_0^{1}(\Omega))^3)^*$ is the dual space of $(H_0^{1}(\Omega))^3$.

The key part of the process `{\it ESTIMATE} ' in adaptive finite element algorithm is to find an a posteriori error estimator $\eta$ and establish the relationship between $\eta$ and the absolute error of solution, ${\bm{e}}={{\bm{u}}-{\bm{u}_h}}:=(e_1,e_2,e_3)\in { {(H_0^1(\Omega) )^3}}$, as following:
\begin{equation}	\label{eq:eta_e}
	\underline{C}\eta\leq {\|{\bm{e}}\|_{\epsilon,\Omega}}\leq \overline{C}\eta,
\end{equation}
with the constants $\underline{C}$ and $\overline{C}$ independent of $\bm{u}_h$ and $\bm{f}$. Here, the lower and upper bounds are called the efficiency and reliability of the a  posteriori error estimator, respectively.


In order to prove inequality \eqref{eq:eta_e}, we follow the strategies for the {a} posteriori error estimation of nonlinear equations in Ref. \cite{Verfuhrtgeneral} and split the process into two steps, which are to determine the relationship between ${\bm F}({\bm{u}_h})$ and $\bm{e}$ in Sec. \ref{sec:Fande} 
and the relationship between ${\bm F}({\bm{u}_h})$ and the estimator $\eta$ in Sec. \ref{sec:Fandeta}. {Furthermore, we show that $\overline{C}$ does not depend on $\epsilon$.}


\subsection{Relationship between ${\vect{F}}({\vect{u}}_h)$ and $\vect{e}$} \label{sec:Fande}
{
Before illustrating the relationship, we show the solution existence and uniqueness of the linear problem \eqref{modelDF}
where the operator construction in Ref. \cite {Jerome1} and the proof of uniqueness in Ref. \cite{Gajewski} are utilized as references.

\begin{lemma}\label{unique1}
If $p$, $n \in L^{\infty}(\Omega)$, $\psi\in W^{1,\infty}(\Omega)$, and
\begin{equation}
	\|\nabla\psi\|_{L^{\infty}(\Omega)}+\sqrt{2}C_p {\epsilon^{-1}}\max\{ \|p\|_{L^{\infty}(\Omega)}, \|n\|_{L^{\infty}(\Omega)}\}< \sqrt{2}/(1+C_p^2),
\end{equation}
there exists a unique solution of  Eq. \eqref{modelDF}. Here, the constant $C_p$ depends on $\Omega$ only.
\end{lemma}
\begin{proof}
	See \ref{app:exist}.
\end{proof}


{
\begin{lemma}\label{regularity1}
{If $p$, $n$, and $\psi$ satisfy the condition in Lemma \ref{unique1},} we have the regularity conclusion for Eq.~\eqref{modelDF} as follows:
 \begin{equation}\label{regularity}
 \epsilon C_1\|{\bm{R}}\|_{{-\epsilon},\Omega}\leq \|{\bm{\phi}}\|_{\epsilon,\Omega}\leq C_2\|{\bm{R}}\|_{{-\epsilon},\Omega},\quad \forall {\bm{R}}\in (L^2(\Omega))^3,
  \end{equation}
where the constants $C_1$ and $C_2$ are independent of ${\bm{\phi}}$, ${\bm{R}}$ and $\epsilon$.
\end{lemma}
\begin{proof}
Eqs. \eqref{vmodel2} and \eqref{definition} lead to the variational form of Eq.~\eqref{modelDF} as
 \begin{align}\label{model2D}
<{\bm{DF}}({\bm{u}}){\bm{\phi}},{\bm{v}}> =&\int_{\Omega}(\nabla\phi_1\cdot\nabla v_1+p\nabla\phi_3\cdot\nabla v_1+\phi_1\nabla\psi\cdot\nabla v_1	+\nabla\phi_2\cdot\nabla v_2\nonumber\\
-n\nabla\phi_3&\cdot\nabla v_2-\phi_2\nabla\psi\cdot\nabla v_2+\epsilon\nabla\phi_3\cdot\nabla v_3-\phi_1v_3+\phi_2v_3)d{\bm x}.	
\end{align}
%

Owing to properties,
\begin{align}
	&\int_\Omega\nabla \phi_1\cdot \nabla v_1d{\bm x}\leq |\phi_1|_{1,\Omega}\cdot {|v_1|_{1,\Omega}}, \\
	&\int_\Omega \phi_1 v_3d{\bm x}\leq \|\phi_1\|_{L^2(\Omega)}\|v_3\|_{L^2(\Omega)},	\\
	&\int_\Omega p\nabla\phi_3\cdot\nabla v_1d{\bm x}\leq \|p\|_{L^{\infty}(\Omega)}\left | \int_\Omega \nabla\phi_3\cdot\nabla v_1d{\bm x}\right |\lesssim |\phi_3|_{1,\Omega}|v_1|_{1,\Omega},	\\
&\int_\Omega \phi_1\nabla\psi\cdot\nabla v_1d{\bm x} \leq  \|\psi\|_{W^{1,\infty}(\Omega)} \int_\Omega(|\phi_1{\partial_x v_1}| +|\phi_1{\partial_y v_1}|)d{\bm x}
 \nonumber\\
&\quad \quad \quad \quad \quad \quad\quad \ \ \lesssim \|\phi_1\|_{L^2(\Omega)}|v_1|_{1,\Omega}
,
\end{align}
we have
\begin{align}
	|<{\bm{DF}}({\bm{u}}){\bm{\phi}},{\bm{v}}>|
	 \lesssim &{|\phi_1|_{1,\Omega}|v_1|_{1,\Omega}+|\phi_3|_{1,\Omega}|v_1|_{1,\Omega}+\|\phi_1\|_{L^2(\Omega)}|v_1|_{1,\Omega}
} \nonumber\\
	 &+|\phi_2|_{1,\Omega}|v_2|_{1,\Omega}+{|\phi_3|_{1,\Omega}|v_2|_{1,\Omega}+\|\phi_2\|_{L^2(\Omega)}|v_2|_{L^2(\Omega)} } \nonumber\\
	 &+\epsilon|\phi_3|_{1,\Omega}|v_3|_{1,\Omega}+\|\phi_1\|_{L^2(\Omega)}\|v_3\|_{L^2(\Omega)}+{\|\phi_2\|_{L^2(\Omega)}\|v_3\|_{L^2(\Omega)}} \nonumber \\
\lesssim &{\epsilon^{-1}\|\bm{\phi}\|_{\epsilon,\Omega}\|\bm{v}\|_{\epsilon,\Omega}},
\end{align}
and hence {$\|\bm{\phi}\|_{\epsilon,\Omega}\geq \epsilon C_1\|\bm{R}\|_{{-\epsilon},\Omega}$}, with $C_1$ independent of ${\bm{\phi}}$, ${\bm{R}}$ and $\epsilon$.

Besides, we have the one-to-one mapping ${\bm{DF}}(\bm{u}): (H_0^1(\Omega))^3\rightarrow (L^2(\Omega))^3$ due to Lemma \ref{unique1}.
By the bounded inverse theorem there exists an inverse
${\bm{DF}}(\bm{u})^{-1}: (L^2(\Omega))^3$ $\rightarrow$ $ (H_0^1(\Omega))^3$,
and by the Hahn-Banach theorem there exists a linear extension $\bm{G}: (H^{-1}(\Omega))^3\rightarrow (H_0^1(\Omega))^3$ satisfying $\bm{G}|_{(L^2(\Omega))^3}={\bm{DF}}(\bm{u})^{-1}$ and $\|\bm{G}\|=\|{\bm{DF}}(\bm{u})^{-1}\|$. Thus,  $\forall \bm{R}\in (L^2(\Omega))^3\subset(H^{-1}(\Omega))^3$,
\begin{equation}
	\|\bm{\phi}\|_{1,\Omega}= \|\bm{G} \bm{R}\|_{1,\Omega}\leq C_2\|\bm{R}\|_{-1,\Omega},
\end{equation}
then, with $\|\bm{v}\|_{\epsilon,\Omega}\leq \|\bm{v}\|_{1,\Omega},\forall \bm{v}\in (H_0^1(\Omega))^3$, we have
\begin{equation}
	\|\bm{\phi}\|_{\epsilon,\Omega}\leq \|\bm{\phi}\|_{1,\Omega}\leq C_2\|\bm{R}\|_{-1,\Omega}\leq C_2\|\bm{R}\|_{-\epsilon,\Omega},
\end{equation}
and Eq. \eqref{regularity} is proved with $C_1$ and $C_2$ independent of $\bm{\phi}$, $\bm{R}$ and $\epsilon$.
\end{proof}
\begin{theorem}\label{key}
{If $p$, $n$ and $\psi$ satisfy the condition in Lemma \ref{unique1},} and
{the error $\bm{e}$ is small enough,} then,
\begin{equation} \label{basic1}
{{2\epsilon C_1C_2/(  C_1+2C_2)}}\|{\bm{F}}({\bm{u}_h})\|_{{-\epsilon},\Omega}\leq \|\bm{e}\|_{{\epsilon},\Omega}\leq 2C_2\|{\bm{F}}({\bm{u}_h})\|_{{-\epsilon},\Omega}.
\end{equation}
Here, $C_1$ and $C_2$ are constants in Eq. \eqref{regularity} as taking $\bm{\phi}=\bm{e}$.
\end{theorem}
\begin{proof}
The definition of $G$-derivative~\cite{Zeidler} indicates that,   $\forall {\bm{w}}\in (H_0^1(\Omega))^3$,
\begin{equation}
	\int_{0}^{1}<{\bm{DF}}(\bm{u}+t\bm{e})\bm{e},\bm{{\bm{w}}}>dt=<{\bm{F}}({\bm{u}_h}),{\bm{w}}>-<{\bm{F}}({\bm{u}}),{\bm{w}}>.
\end{equation}
By means of the fact that $<{\bm{F}}({\bm{u}}),{\bm{w}}>=0,\ \forall {\bm{w}}\in (H_0^1(\Omega))^3$, we have
\begin{equation}\label{proofth1basic}
<{\bm{F}}({\bm{u}_h}),{\bm{w}}> = \int_{0}^{1}<{\bm{DF}}({\bm{u}}+t{\bm{e}}){\bm{e}},{\bm{w}}>dt,
\end{equation}
and thus,
\begin{equation}\label{refine}
{ <{\bm{DF}}({\bm{u}}){\bm{e}},{\bm{w}}>=\int_{0}^{1} <{\bm{DF}}({\bm{u}}){\bm{e}}-{\bm{DF}}({\bm{u}}+t{\bm{e}}){\bm{e}},{\bm{w}}>dt+<{\bm{F}}({\bm{u}_h}),{\bm{w}}> .}
\end{equation}
{We now define $<\bm{DF}({\bm{u}}){\bm{e}},{\bm{w}}>=<\widetilde{\bm{R}},{\bm{w}}>$ for convenience. Then, the right inequality in \eqref{regularity} leads to {$\|{\bm{e}}\|_{\epsilon,\Omega}\leq C_2||\widetilde{\bm{R}}||_{{-\epsilon},\Omega}$} as taking $\bm{\phi}=\bm{e}$.}
Next, the integral part in Eq. \eqref{refine} is estimated, i.e.,
\begin{align}
&|<{\bm{DF}}({\bm{u}}){\bm{e}}-{\bm{DF}}({\bm{u}}+t{\bm{e}}){\bm{e}},{\bm{w}}>|	\nonumber\\
=	&	2t\left|\int_{\Omega}(e_1\nabla e_3\cdot\nabla w_1-e_2\nabla e_3\cdot\nabla w_2)d{\bm x}\right|	\nonumber\\
\leq	& 2t\max_{i=1,2}\{\|e_i\|_{L^{\infty}(\Omega)}\}(||\nabla e_3||_{L^2(\Omega)} ||\nabla w_1||_{L^2(\Omega)} + ||\nabla e_3||_{L^2(\Omega)} ||\nabla w_2||_{L^2(\Omega)})		\nonumber\\
\leq	& 2\sqrt{2} t\epsilon^{-1}\max_{i=1,2}\{\|e_i\|_{L^{\infty}(\Omega)}\}	
\|{\bm{e}}\|_{\epsilon,\Omega}\|{\bm{w}}\|_{\epsilon,\Omega},\label{3.4}
\end{align}
and hence
%
\begin{align}
&\|{\bm{e}}\|_{\epsilon,\Omega} \nonumber\\
\leq& C_2||{\widetilde{\bm{R}}}||_{{-\epsilon},\Omega}\nonumber\\
\leq& C_2\sup\limits_{\|{\bm{w}}\|_{\epsilon,\Omega}=1}\left(\left|\int_{0}^{1} <{\bm{DF}}({\bm{u}}){\bm{e}}-{\bm{DF}}({\bm{u}}+t{\bm{e}}){\bm{e}},{\bm{w}}>dt\right|+|<{\bm{F}}({\bm{u}_h}),{\bm{w}}>|\right)\nonumber\\
\leq& C_2\sup\limits_{\|{\bm{w}}\|_{\epsilon,\Omega}=1}\left(\sqrt{2}\epsilon^{-1}\max_{i=1,2}\{\|e_i\|_{L^{\infty}(\Omega)}\}
\|{\bm{e}}\|_{\epsilon,\Omega}\|{\bm{w}}\|_{\epsilon,\Omega}\right)+C_2\|F({\bm{u}_h})\|_{{-\epsilon},\Omega} \nonumber\\
=&{\sqrt{2}\epsilon^{-1}C_2\max_{i=1,2}\{\|e_i\|_{L^{\infty}(\Omega)}\}\|{\bm{e}}\|_{\epsilon,\Omega}}+C_2\|F({\bm{u}_h})\|_{{-\epsilon},\Omega}.
\end{align}
Notably, we shall have $\|{\bm{e}}\|_{\epsilon,\Omega}\lesssim \|{\bm{F}}({\bm{u}_h})\|_{{-\epsilon},\Omega}$ if $\max\limits_{i=1,2}\{\|e_i\|_{L^{\infty}(\Omega)}\}< \epsilon/\left(\sqrt{2}C_2\right)$, however, we restrict the condition as $\max\limits_{i=1,2}\{\|e_i\|_{L^{\infty}(\Omega)}\}\leq \epsilon/(2\sqrt{2}C_2)$ for convenience, and obtain
\begin{equation}
	\|{\bm{e}}\|_{{\epsilon},\Omega}\leq 2C_2 \|{\bm{F}}({\bm{u}_h})\|_{{-\epsilon},\Omega}.
\end{equation}
This proves the second inequality in \eqref{basic1}.

On the other hand, we rewrite Eq. $\eqref{proofth1basic}$ as
\begin{equation}
	<{\bm{F}}({\bm{u}_h}),{\bm{w}}> =  <{\bm{DF}}({\bm{u}}){\bm{e}},{\bm{w}}> -\int_{0}^{1} <{\bm{DF}}({\bm{u}}){\bm{e}}-{\bm{DF}}({\bm{u}}+t{\bm{e}}){\bm{e}},{\bm{w}}>dt,
\end{equation}
thereby,
\begin{align}
	\|{\bm{F}}({\bm{u}_h})\|_{{-\epsilon},\Omega} \leq & \|{\bm{DF}}({\bm{u}}){\bm{e}}\|_{{-\epsilon},\Omega} +\sqrt{2}\max_{i=1,2}\{\|e_i\|_{L^{\infty}(\Omega)}\}\epsilon^{-1}\|{\bm{e}}\|_{\epsilon,\Omega} \nonumber \\
	\leq & 1/(\epsilon C_1)\|{\bm{e}}\|_{\epsilon,\Omega} + \sqrt{2}\max_{i=1,2}\{\|e_i\|_{L^{\infty}(\Omega)}\}\epsilon^{-1}\|{\bm{e}}\|_{\epsilon,\Omega} \nonumber\\
	\leq & {[1/(\epsilon C_1) + 1/(2C_2)]\|{\bm{e}}\|_{\epsilon,\Omega}},
\end{align}
where  the second inequality is given by using the left inequality in \eqref{regularity}. Consequently, we have
\begin{equation}
{	 \|{\bm{e}}\|_{\epsilon,\Omega} \geq 2\epsilon C_1C_2/\left( C_1+2C_2\right)\|{\bm{F}}({\bm{u}_h})\|_{{-\epsilon},\Omega}},
\end{equation}
which proves the first inequality in \eqref{basic1}.

\end{proof}


\subsection{Relationship between ${\vect{F}}({\vect{u}}_h)$ and the a posteriori error estimator $\eta$} \label{sec:Fandeta}

First, some notations are given as follows.
{For a regular triangle subdivision $\mathcal{T}_h$ of $\Omega$, $\mathcal{N}_h$ represents the set of all vertices divided,  $\mathcal{E}_{h}$ represents all edges contained in $\mathcal{T}_h$, and $\mathcal{I}_h=\mathcal{E}_{h}\backslash\partial\Omega$ contains the inner edges of $\mathcal{T}_h$.
We set $\widetilde{w}_T=\bigcup_{x_i\in T}\Omega_i$ and $\widetilde{w}_E=\bigcup_{x_i\in E}\Omega_i$, where $\Omega_i = \{T\in \mathcal{T}_h,\ x_i\in T\}$}. $h_B=\textup{diam}(B)$ denotes the diameter of any set $B$.
Let  $E$ be the shared edge of $T$ and $T'$, i.e., $E = T\cap T'$, and $\bm{n}_E$ represent the outward normal vector of $E$ in $T$, we define the jump across the edge by
\begin{equation}
	[\nabla v\cdot \bm{n}_E]:=\nabla v\cdot \bm{n}_E|_{T}-\nabla v\cdot \bm{n}_E|_{T'} ,\quad\forall v\in H_0^1(\Omega),
\end{equation}
and $\forall E\in\partial\Omega$, we set $[\nabla v\cdot \bm{n}_E]=0$ for convenience.

Next, we define ${\widetilde{\bm{F}}}(\bm{u}_h)$ by, $\forall \bm{v}=(v_1,v_2,v_3)\in (H_0^1(\Omega))^3$,
\begin{align}
	<{\widetilde{\bm{F}}}(\bm{u}_h),\bm{v}>=\int_{\Omega}(\nabla p_h\cdot\nabla v_1+p_h\nabla\psi_h\cdot\nabla v_1
+\nabla n_h \cdot\nabla v_2-n_h\nabla\psi_h\cdot \nabla v_2 \nonumber\\
+\epsilon\nabla\psi_h\cdot \nabla v_3-p_hv_3+n_hv_3)d\bm{x}-\sum\limits_{T\in \mathcal{T}_h}\int_{T}(f_{T,1}v_1+f_{T,2}v_2+f_{T,3}v_3)d\bm{x},	\label{posteriori1}
\end{align}	
with the mean value of $f_i$ over $T$ being $f_{T,i}=\int_{T}f_i d\bm{x} / |T|,\  i=1,2,3$. Here, $|T|$ denotes the area of $T$.

Let $\lambda_{T,i}$ ($i=1,2,3$) be the area coordinates of the reference element $T$, we define the  bubble functions $b_T$ and $b_E$ as follows,
\begin{align}
	b_T(\bm{x})
&=\begin{cases}
\begin{array}{cl}
27\lambda_{T,1}\lambda_{T,2}\lambda_{T,3},\ \ & \bm{x} \in T,\\
0,\ \ & \bm{x}\in \Omega\backslash T,
\end{array}
\end{cases}	\\
b_{E}(\bm{x})
&=\begin{cases}
\begin{array}{cl}
4\lambda_{T_1,j}\lambda_{T_1,k},\ \ &\bm{x}\in T_1,\\
4\lambda_{T_2,l}\lambda_{T_2,m},\ \ &\bm{x}\in T_2,\\
0,\ \ & \bm{x}\in \Omega\backslash {w}_E,
\end{array}
\end{cases}
\end{align}
where $j$ and $k$ are the indices of $E$'s two vertexes associated with $T_1$ while $l$ and $m$ are those with $T_2$, and $w_E= T_1\cup T_2$.
The space of vector bubble functions is then denoted by
${\bm{\widetilde{Y}}}_h:=(\widetilde{Y}_h^0)^3$
with
$\widetilde{Y}_h^0 =\textup{span}\{ b_T u,\ b_EP w : \forall u\in \mathcal{P}_1(T),\ \forall w\in\mathcal{P}_1(E),\ \forall T\in\mathcal{T}_h,\ \forall E\in\mathcal{I}_{h}\}$.
Here, $\mathcal{P}_1(T)$ and $\mathcal{P}_1(E)$ are spaces of linear polynomials on $T$ and $E$, respectively, and $P: L^\infty(E)\rightarrow L^\infty(T)$ is a continuation operator.

Two lemmas are now given before showing Theorems \ref{th2} and \ref{th3}.



\begin{lemma}
Let  {$R_h$} be the Scott-Zhang interpolation operator~\cite{Scootzhang} for a regular partition,
then,
\begin{align}
\|v-R_hv\|_{L^2(T)} &\leq C_3 h_T\epsilon^{-1/2} \|v\|_{\epsilon,\widetilde{w}_T},\quad \forall T \in\mathcal{T}_h,\label{eq:SZ1} \\
\|v-R_hv\|_{L^2(E)} &\leq C_4  h_E^{1/2}\epsilon^{-1/2}\|v\|_{\epsilon,\widetilde{w}_E},\quad \forall E\in \mathcal{E}_{h}, \label{eq:SZ2}\\
{\|R_h v\|_{L^2(T)}}&\leq C_5 h_T\epsilon^{-1/2}\|v\|_{\epsilon,\widetilde{w}_T},\quad \forall T \in\mathcal{T}_h,\label{eq:SZ3}\\
{\|R_h v\|_{L^2(E)}}&\leq C_6 h_E^{1/2}\epsilon^{-1/2}\|v\|_{\epsilon,\widetilde{w}_E},\quad \forall E\in \mathcal{E}_{h},\label{eq:SZ4}
\end{align}
where constants $C_3$, ..., $C_6$ only depend on the reference element and regular partition.
\end{lemma}

\begin{proof}

{Inequalities \eqref{eq:SZ1} and \eqref{eq:SZ2}  follow the estimation of Scoot-Zhang interpolation in~\cite{Scootzhang} and $|v|_{1,\Omega}\leq \epsilon^{-1/2} \|v\|_{\epsilon,T}$}. Inequalities \eqref{eq:SZ3} and \eqref{eq:SZ4} are proved similarly with $H\ddot{o}lder$ inequality and trace theorem.

\end{proof}

\begin{lemma}\label{paopao}
(Bubble function space~\cite{Oden,Verfuhrt})
$\forall {u}\in \mathcal{P}_1(T)$ and $\forall {w} \in \mathcal{P}_1(E)$ where  $T\in\mathcal{T}_h$ and $E\in\mathcal{I}_{h}$,
\begin{align}
&{\widetilde{C}_{1}\|{{u}}\|_{L^2(T)}\leq \|b_Tu\|_{L^2(T)}\leq \|{{u}}\|_{L^2(T)}}\label{paopao0}\\
& \widetilde{C}_2\|{{u}}\|_{L^2(T)}\leq \sup\limits_{{{v}}\in\mathcal{P}_1(T)}\frac{\int_T {{u}}b_T{{v}}d\bm{x}}{\|{{v}}\|_{L^2{(T)}}}\leq \|{{u}}\|_{L^2(T)}, \label{paopao1}\\
& \widetilde{C}_3\|w\|_{L^2(E)}\leq \sup\limits_{{{\tau}}\in\mathcal{P}_1(E)}\frac{\int_E w b_EP{{\tau}}ds}{\|{{\tau}}\|_{L^{2}(E)}}\leq \|w\|_{L^2(E)}, \label{paopao2}\\
&\widetilde{C}_4 h_T^{-1}\|b_T {{u}}\|_{L^2(T)}\leq\|\nabla(b_T {{u}})\|_{L^2(T)}\leq \widetilde{C}_5 h_T^{-1}\|b_T {{u}}\|_{L^2(T)},\label{paopao3}\\
&\widetilde{C}_6h_T^{-1}\|b_{E}Pw\|_{L^2(T)}\leq\|\nabla(b_{E}Pw)\|_{L^2(T)}\leq \widetilde{C}_{7}h_T^{-1}\|b_{E}Pw\|_{L^2(T)},	\label{paopao4}\\
& \|\nabla{{(b_T u)}}\|_{L^2(T)}\leq {\widetilde{C}_8 }\|u\|_{1,T}, \label{paopao00}\\
& \|b_{E}{Pw}\|_{L^2(T)}\leq \widetilde{C}_{9}h_T^{{1}/{2}}\|w\|_{L^2(E)},	\label{paopao5}
\end{align}
{where constants $\widetilde{C}_1,..., \widetilde{C}_{9}$} depend on the reference element and regular partition only.
\end{lemma}
\begin{proof}
The inequality \eqref{paopao0} is from Theorem 2.2 in Ref. \cite{Oden}. Inequalities (\ref{paopao1}-\ref{paopao4}) and \eqref{paopao5} 
are given in Lemma 5.1 of Ref. \cite{Verfuhrt}. The proof of inequality (\ref{paopao00}) is presented as follows,
\begin{align}
\|\nabla(b_Tu)\|_{L^2(T)}^2
=&\int_T\Big[ (\partial_x{b_{T}} u+b_T\partial_xu)^2+  (\partial_y{b_{T}}u+b_T\partial_yu)^2\Big]d{\bm x} \nonumber\\
\leq & 2\int_T\Big[ (\partial_x{b_{T}}u)^2+(b_T\partial_xu)^2+  (\partial_y{b_{T}}u)^2+(b_T \partial_y u)^2\Big]d{\bm x}\nonumber\\
\leq &2\max\limits_{\bm{x}\in T}\left\{(\partial_x{b_{T}})^2,({\partial_yb_{T}})^2,1 \right \}\int_T[2 u^2+(\partial_xu)^2+(\partial_yu)^2]d{\bm x}\nonumber\\
\leq & \widetilde{C}_8^2\|u\|_{1,T}^2,	
\end{align}
where $|b_T|<1$ is applied for showing the second inequality.
\end{proof}

\begin{theorem}\label{th2}
{There exists  a {constant $C_7$} independent of $\bm{u}_h$, $\bm{f}$ and $\epsilon$,}
such that,
\begin{equation}	\label{fuh_eta}
\|{\bm{F}}({\bm{u}_h})\|_{{-\epsilon},\Omega}\leq C_7{(}\eta +\varepsilon{)},
\end{equation}
where the estimator $\eta:= \left(\sum\limits_{T\in\mathcal{T}_h} \eta_T^2+ \sum\limits_{E\in \mathcal{I}_h} \eta_E^2 \right)^{{1}/{2}}$ and the oscillation term $\varepsilon:= \left(\sum\limits_{T\in\mathcal{T}_h} \varepsilon_T^2\right)^{{1}/{2}}$, with
\begin{align}
\eta_T^2:=& {h_T^2\epsilon^{-1}}\left(|| \nabla\cdot(p_h\nabla \psi_h) +f_{T,1}||_{L^2(T)}^2 + || \nabla\cdot(n_h\nabla\psi_h)-f_{T,2} ||_{L^2(T)}^2\right.\nonumber\\
&\left.+|| n_h-p_h-f_{T,3}||_{L^2(T)}^2\right),  \nonumber\\
\eta_E^2:= &{h_E\epsilon^{-1}}\left(||  [ \nabla p_h\cdot \bm{n}_E]+[p_h\nabla\psi_h\cdot \bm{n}_E] \parallel_{{L^2(E)}}^2+|| [ \nabla n_h \cdot \bm{n}_E]-[n_h\nabla\psi_h\cdot \bm{n}_E]||_{{L^2(E)}}^2\right. \nonumber\\
&\left.+|| [\nabla\psi_h\cdot \bm{n}_E] ||_{{L^2(E)}}^2\right),	\nonumber\\
	\varepsilon^2_T:= &{h_T^2\epsilon^{-1}}\sum\limits_{i=1}^3||f_i-f_{T,i}||_{L^2(T)}^2.\nonumber
\end{align}
\end{theorem}
\begin{proof}
{We define $\bm{R}_h \bm{\phi}:=(R_h\phi_1,R_h\phi_2,R_h\phi_3)$, $\forall {\bm{\phi}}=(\phi_1,\phi_2,\phi_3)\in (H_0^1(\Omega))^3$, with $R_h$ a Scott-Zhang interpolation and have $\bm{R}_h \in \mathcal{L}((H_0^1(\Omega))^3,(V_h)^3)$. With $ {\bm{Y}}=(H_0^1(\Omega))^3$, we denote by ${\bm{Y}}_h^*$,${\bm{R}_h^*}$ and $(\bm{\textup{Id}_{Y}}-\bm{R}_h)^*$ the dual spaces of ${\bm{Y}}_h= (V_h)^3$, the dual operator of ${\bm{R}_h}$ and the dual operator of $(\bm{\textup{Id}_{Y}}-\bm{R}_h)$, respectively. {$\mathcal{L}({\bm{Y}},{\bm{Y}}_h)$ is a Banach space of continuous linear maps of $\bm{Y}$ in $\bm{Y}_h$}.
Following the strategy in Ref.~\cite{Verfuhrt}}, we decompose the inner product into four parts as follows,
\begin{align} <{\bm{F}}({\bm{u}_h}),{\bm{\phi}}>=&<{\bm{\widetilde{F}}}({\bm{u}_h}),{\bm{\phi}}-\bm{R}_h{\bm{\phi}}>+<{\bm{F}}({\bm{u}_h})
-{\bm{\widetilde{F}}}({\bm{u}_h}),{\bm{\phi}}-\bm{R}_h{\bm{\phi}}> \nonumber\\
&\ +<{\bm{\widetilde{F}}}({\bm{u}_h}),\bm{R}_h{\bm{\phi}}>+<{\bm{F}}({\bm{u}_h})-{\bm{\widetilde{F}}}({\bm{u}_h}),\bm{R}_h{\bm{\phi}}>,
\end{align}
and thus,
\begin{align}
	\|{\bm{F}}({\bm{u}_h})\|_{{-\epsilon},\Omega} 
	\leq &\|(\bm{\textup{Id}_{Y}}-\bm{R}_h)^*{\bm{\widetilde{F}}}({\bm{u}_h})\|_{{-\epsilon},\Omega}+\|(\textup{Id}_{\bm{Y}}-\bm{R}_h)^*({\bm{F}}({\bm{u}_h})-{\bm{\widetilde{F}}}({\bm{u}_h}))\|_{{-\epsilon},\Omega}\nonumber\\
&+{\|\bm{R}_h^*({\bm{\widetilde{F}}}({\bm{u}_h}))\|_{{-\epsilon},\Omega}}+
{\|\bm{R}_h^* ({\bm{F}}({\bm{u}_h})
-{\bm{\widetilde{F}}}({\bm{u}_h}))\|_{{-\epsilon},\Omega}}. \label{eq:Fuh1}
\end{align}
The four terms on the right hand side (RHS) of the inequality \eqref{eq:Fuh1} are discussed one by one in the following.

The first term is
\begin{align}
&\parallel  (\bm{Id_Y}-\bm{R_h})^*{\bm{\widetilde{F}}}({\bm{u}_h})\parallel_{{-\epsilon},\Omega} \nonumber \\
=&\sup\limits_{{\bm{\phi}}\in {\bm{Y}},\|{\bm{\phi}}\|_{\epsilon,\Omega}=1}{ \Big|} \sum\limits_{{T\in\mathcal{T}_h}}\Big{\{}\int_T [(-\Delta p_h-\nabla\cdot(p_h\nabla \psi_h)-f_{T,1})(\phi_1-R_h\phi_{1})	\nonumber\\
&+(-\Delta n_h+\nabla\cdot(n_h\nabla\psi_h)-f_{T,2})(\phi_2-R_h\phi_{2})	\nonumber\\
&+(-\Delta\psi_h-p_h+n_h-f_{T,3})(\phi_3-R_h\phi_{3})] d{\bm x} \nonumber\\
&+\sum\limits_{E\in \mathcal{I}_h \cap \partial T}\int_{E}[([\nabla p_h\cdot {\bm n}_E]+[ p_h\nabla\psi_h\cdot \bm{n}_E])(\phi_1-R_h\phi_{1}) 	\nonumber\\
&+([\nabla n_h\cdot \bm{n}_E]-[ n_h\nabla\psi_h \cdot \bm{n}_E])(\phi_2-R_h\phi_{2})+[\nabla\psi_h\cdot \bm{n}_E](\phi_3-R_h\phi_{3})  ]ds\Big{\}}{ \Big|} 	\nonumber\\
\lesssim &  \sup\limits_{{\bm{\phi}} \in {\bm{Y}},\|{\bm{\phi}}\|_{\epsilon,\Omega}=1}\Big{\{} \sum\limits_{{T\in\mathcal{T}_h}}h_T\epsilon^{-1/2}
\left[\| \nabla\cdot(p_h\nabla \psi_h) +f_{T,1}\|_{L^2(T)}\|\phi_1\|_{\epsilon,\widetilde{w}_{T}}	\right.\nonumber\\
&\left. +\| \nabla\cdot (n_h\nabla\psi_h) -f_{T,2}\|_{L^2(T)}\|\phi_2\|_{\epsilon,\widetilde{w}_{T}}+\|  n_h-p_h-f_{T,3}\|_{L^2(T)}\|\phi_3\|_{\epsilon,\widetilde{w}_{T}} \right] \nonumber\\
&+\sum\limits_{E\in \mathcal{I}_h}h_E^{1/2}\epsilon^{-1/2}\left({\|  [\nabla p_h \cdot \bm{n}_E]+[p_h\nabla\psi_h\cdot \bm{n}_E] \|_{L^2(E)}}\|\phi_1\|_{\epsilon,\widetilde{w}_{E}} \right.\nonumber\\
 &\left.+\parallel [\nabla n_h \cdot \bm{n}_E]-[n_h\nabla\psi_h\cdot \bm{n}_E] \parallel_{L^2(E)}
\|\phi_2\|_{\epsilon,\widetilde{w}_{E}}+\parallel [\nabla\psi_h\cdot \bm{n}_E] \parallel_{L^2(E)}\|\phi_3\|_{\epsilon,\widetilde{w}_{E}} \right)
\Big{\}}  \nonumber\\
\lesssim &  \sup\limits_{{\bm{\phi}} \in {\bm{Y}},\|{\bm{\phi}}\|_{\epsilon,\Omega}=1}\Big{\{} \sum\limits_{{T\in\mathcal{T}_h}}
\eta_T\left(\|\phi_1\|^2_{\epsilon,\widetilde{w}_{E}}+\|\phi_2\|^2_{\epsilon,\widetilde{w}_{E}}+\|\phi_3\|^2_{\epsilon,\widetilde{w}_{E}}\right)^{1/2}\nonumber\\
&+\sum\limits_{E\in \mathcal{I}_h}\eta_E\left(\|\phi_1\|^2_{\epsilon,\widetilde{w}_{E}}+\|\phi_2\|^2_{\epsilon,\widetilde{w}_{E}}+\|\phi_3\|^2_{\epsilon,\widetilde{w}_{E}}\right)^{1/2}
\Big{\}}  \nonumber\\
\lesssim & \eta,
\end{align}
where the first inequality is shown by combining inner edges of any triangle $T$ and using inequalities~\eqref{eq:SZ1} and \eqref{eq:SZ2}. Cauchy-Schwartz inequality is utilized for giving the second and final inequalities.

Similarly, the second term on the RHS of  \eqref{eq:Fuh1} is
\begin{align}
&\|(\bm{Id_Y}-\bm{R_h})^*[{\bm{F}}({\bm{u}_h})-{\bm{\widetilde{F}}}({\bm{u}_h})]\|_{{-\epsilon},\Omega}\nonumber \\
=&\sup\limits_{{\bm{\phi}}\in {\bm{Y}}, \|{\bm{\phi}}\|_{\epsilon,\Omega}=1}{ \Big|} \sum\limits_{T\in\mathcal{T}_h}\int_{T} \left[\left(f_1-f_{T,1}\right)\left(\phi_1-R_h\phi_1\right)
+(f_2-f_{T,2})(\phi_2-R_h\phi_2) \right.\nonumber\\
&\left.+(f_3-f_{T,3})(\phi_3-R_h\phi_3)\right]d{\bm x}{ \Big|} \nonumber\\
 \lesssim &  \sup\limits_{{\bm{\phi}}\in {\bm{Y}}, \|{\bm{\phi}}\|_{\epsilon,\Omega}=1} \sum\limits_{T\in\mathcal{T}_h}h_T\epsilon^{-1/2}\left(\|f_1-f_{T,1}\|_{L^2(T)}\|{{\phi}_1}\|_{\epsilon,\widetilde{w}_{T}}+
\|f_2-f_{T,2}\|_{L^2(T)}\|{{\phi}}_2\|_{\epsilon,\widetilde{w}_{T}} \right.\nonumber\\
&\left.+\|f_3-f_{T,3}\|_{L^2(T)}\|{{\phi}_3}\|_{\epsilon,\widetilde{w}_{T}}\right) \nonumber\\
\lesssim & \varepsilon .	 
\end{align}

In the third term on the RHS of \eqref{eq:Fuh1},
\begin{align}
&\|\bm{R}_h^*({\bm{\widetilde{F}}}({\bm{u}_h}))\|_{{-\epsilon},\Omega} \nonumber\\
=&\sup\limits_{{{\bm{\phi}}\in \bm{Y};\|{\bm{\phi}}\|_{\epsilon,\Omega}=1}} { \Big|}\sum\limits_{T\in\mathcal{T}_h}\Big{\{}\int_{T}[(-\Delta p_h-\nabla\cdot(p_h\nabla \psi_h)-f_{T,1}){R_h\phi_1}  \nonumber\\
&	+(-\Delta n_h+\nabla\cdot(n_h\nabla\psi_h)-f_{T,2})R_h\phi_2	+(-\Delta\psi_h-p_h+n_h-f_{T,3})R_h\phi_3]d\bm{x}	 \nonumber\\
&+\sum\limits_{E\in \mathcal{I}_h \cap \partial T }\int_{E} [\left([\nabla p_h\cdot \bm{n}_E]+[ p_h\nabla\psi_h\cdot \bm{n}_E]\right)R_h\phi_1+([\nabla n_h\cdot \bm{n}_E]-[ n_h\nabla\psi_h \cdot \bm{n}_E])R_h\phi_2 	 \nonumber\\
& +[\nabla\psi_h\cdot \bm{n}_E]R_h\phi_3] ds  \Big{\}}{ \Big|}		 \nonumber\\
\lesssim&  \sup\limits_{{\bm{\phi}}\in \bm{Y};\|{\bm{\phi}}\|_{\epsilon,\Omega}=1}	\Big{\{}\sum\limits_{T\in\mathcal{T}_h}
{h_{T}\epsilon^{-1/2}}\left(|| \nabla\cdot(p_h\nabla \psi_h )+f_{T,1}||_{L^2(T)}\|\phi_1\|_{\epsilon,\widetilde{w}_T}\right.	 \nonumber\\
&\left.+|| \nabla\cdot(n_h\nabla\psi_h) -f_{T,2} ||_{L^2(T)}\|\phi_2\|_{\epsilon,\widetilde{w}_T}	+||n_h-p_h-f_{T,3}||_{L^2(T)}\|\phi_3\|_{\epsilon,\widetilde{w}_T}\right)		 \nonumber\\
&	+ \sum\limits_{E\in \mathcal{I}_h }{h_E^{1/2}\epsilon^{-1/2}}\left(||[\nabla p_h \cdot \bm{n}_E]+[ p_h\nabla\psi_h\cdot \bm{n}_E] ||_{L^2(E)}\|\phi_1\|_{\epsilon,\widetilde{w}_E}  \right.\nonumber\\
&\left.	+ || [\nabla n_h\cdot \bm{n}_E]-[ n_h\nabla\psi_h\cdot \bm{n}_E] ||_{L^2(E)}
\|\phi_2\|_{\epsilon,\widetilde{w}_E}	 + \|[\nabla\psi_h\cdot \bm{n}_E] \|_{L^2(E)}\|\phi_3\|_{\epsilon,\widetilde{w}_E} \right)\Big{\}}		 \nonumber\\
\lesssim &  \eta,	\label{Fyh}
\end{align}
To prove the first inequality in \eqref{Fyh} we need use inequalities \eqref{eq:SZ3} and \eqref{eq:SZ4}.

At last, we have in the fourth term on the RHS of \eqref{eq:Fuh1} that,
\begin{align}
&{\|\bm{R}_h^* ({\bm{F}}({\bm{u}_h})
-{\bm{\widetilde{F}}}({\bm{u}_h}))\|_{{-\epsilon},\Omega}}	\nonumber\\
=&\sup\limits_{{\bm{\phi}}\in \bm{Y};\|{\bm{\phi}}\|_{\epsilon,\Omega}=1}{ \Big|}\sum\limits_{T\in\mathcal{T}_h}\int_{T}[(f_1-f_{T,1})R_h\phi_1
+(f_2-f_{T,2})R_h\phi_2+(f_3-f_{T,3})R_h\phi_3]d{\bm x}	{ \Big|}\nonumber\\
\lesssim &  \sup\limits_{{\bm{\phi}}\in \bm{Y};\|{\bm{\phi}}\|_{\epsilon,\Omega}=1}\sum\limits_{T\in\mathcal{T}_h}h_{T}\epsilon^{-1/2} \left( \|f_1-f_{T,1}\|_{L^2(T)}\|\phi_1\|_{\epsilon,\widetilde{w}_T}+\|f_2-f_{T,2}\|_{L^2(T)}\|\phi_2\|_{\epsilon,\widetilde{w}_T}	\right. \nonumber\\
+&\left. \|f_3-f_{T,3}\|_{L^2(T)}\|\phi_3\|_{\epsilon,\widetilde{w}_T} \right)\nonumber\\
\lesssim & \varepsilon,	\label{3.22}
\end{align}
{where we have applied inequalities \eqref{eq:SZ3} to obtain the first inequality}. 

Therefore, the results with four terms lead to the conclusion
\begin{equation}
	\|{\bm{F}}({\bm{u}_h})\|_{{-\epsilon},\Omega}\leq C_7(\eta +\varepsilon),
\end{equation}
with $C_7$ independent of $\bm{u}_h$, $\bm{f}$ and $\epsilon$.
\end{proof}

\begin{theorem}\label{th3}
{There exist {constants $C_8$ and $C_9$} independent of ${\bm{u}_h}$, $\bm{f}$} and $\epsilon$, such that
\begin{align}
&\eta_T\leq {C_8(\epsilon^{-1/2}}\|{\bm F}({\bm{u}_h})\|_{-\epsilon,T} +\varepsilon_T),\quad\forall T\in\mathcal{T}_h, \label{th3111} \\
&\eta_E\leq {C_9(\epsilon^{-1/2}}\|{\bm F}({\bm{u}_h})\|_{-\epsilon,\widetilde{w}_E} +\varepsilon_{\widetilde{w}_E}),\quad\forall E\in\mathcal{I}_h, \label{th3112}
\end{align}
where $\varepsilon_{\widetilde{w}_E}:= \big(\sum\limits_{T\in\widetilde{w}_E} \varepsilon_T^2\big) ^{{1}/{2}}$.
\end{theorem}
\begin{proof}
The relationship between $\bm{\widetilde{F}}({\bm{u}_h})$ and $\eta_T$ is considered first, i.e., we have
\begin{align}
&h_{T}|| \nabla\cdot(p_h\nabla \psi_h)+f_{T,1} ||_{L^2(T)} \nonumber \\
= & h_{T}\|\Delta p_h+\nabla\cdot(p_h\nabla \psi_h)+f_{T,1}\|_{L^2(T)}	 \nonumber\\
\lesssim  & h_{T}\sup\limits_{\phi\in\mathcal{P}_1(T)\setminus \{0\}}|<\Delta p_h+\nabla\cdot(p_h\nabla \psi_h)+f_{T,1},b_T\phi>{ _T}|\cdot\|\phi\|_{L^2(T)}^{-1}	 \nonumber\\
\lesssim&  \sup\limits_{\phi\in\mathcal{P}_1(T)\setminus \{0\}}|<\Delta p_h+\nabla\cdot(p_h\nabla \psi_h)+f_{T,1},b_T\phi>{ _T}|\cdot |{b_T\phi}|_{1,T}^{-1}	 \nonumber\\
\lesssim &  \sup\limits_{{\phi\in\mathcal{P}_1(T)\setminus \{0\}}}|<{\bm{\widetilde{F}}}({\bm{u}_h}),(b_T\phi,0,0)>{ _T}|{\cdot |{b_T\phi}|_{1,T}^{-1}	} \nonumber\\
\lesssim &  \sup\limits_{{\bm{v}}\in {{{{\bm{\widetilde{Y}}}_h|_T}}},\|{\bm{v}}\|_{1,{T}}=1}|<{\bm{\widetilde{F}}}({\bm{u}_h}),{\bm{v}}>{_T}|	 \nonumber\\
\lesssim &\|{\bm{\widetilde{F}}}({\bm{u}_h})\|{_{({\bm{\widetilde{Y}}}_h|_T)^*}}	\label{posterioripaopao}
\end{align}
where~\eqref{paopao1} is applied for having the first inequality, and inequalities \eqref{paopao0} and \eqref{paopao3} are used for getting the second inequality.
%

Analogously, we have
\begin{align}
		&h_{T}\| \nabla\cdot(n_h\nabla\psi_h)-f_{T,2} \|_{L^2(T)} \lesssim \|{\bm{\widetilde{F}}}({\bm{u}_h})\|{_{({\bm{\widetilde{Y}}}_h|_T)^*}}, \\
   		&h_{T}\| n_h-p_h-f_{T,3}\|_{L^2(T)} \lesssim \|{\bm{\widetilde{F}}}({\bm{u}_h})\|{_{({\bm{\widetilde{Y}}}_h|_T)^*}}.
\end{align}

On the side of $\eta_E$, ${\forall E\in\mathcal{I}_h}$, we have
\begin{align}
&h_E^{{1}/{2}}||  [\nabla p_h\cdot \bm{n}_E]+[p_h\nabla\psi_h\cdot \bm{n}_E] ||_{L^2(E)}	\nonumber	\\
\lesssim& h_E^{{1}/{2}}\sup\limits_{\delta\in \mathcal{P}_{ 1}(E)\setminus\{0\}}|| { \delta} ||_{L^2(E)}^{-1}\left|\int_{E}([\nabla p_h\cdot \bm{n}_E]+[ p_h\nabla\psi_h\cdot \bm{n}_E])b_EP\delta ds	\right|	 \nonumber\\
\lesssim & h_E^{{1}/{2}}h_T^{{1}/{2}}\sup\limits_{\delta\in \mathcal{P}_{ 1}(E)\setminus\{0\}}||  b_E { P}{\delta}||_{L^2(T)}^{-1}\left|\int_{E}([\nabla p_h\cdot \bm{n}_E]+[p_h\nabla\psi_h\cdot \bm{n}_E])b_EP\delta ds\right|		\nonumber\\
\lesssim& h_T\sup\limits_{\delta\in \mathcal{P}_{ 1}(E)\setminus\{0\}}||  b_E{ P} {\delta}||_{L^2(T)}^{-1}\left|<{\bm{\widetilde{F}}}({\bm{u}_h}),( b_E P\delta,0,0)>{_{\widetilde{w}_E}} \right.		\nonumber\\
 &\left.-\int_{\widetilde{w}_E}(  -\Delta p_h-\nabla\cdot(p_h\nabla \psi_h)-f_{T,1} )b_E P\delta d{\bm x}  \right|			\nonumber\\
\lesssim& \sup\limits_{{\bm{v}}\in{ {\bm{\widetilde{Y}}}_{h}{|_{\widetilde{w}_E}}},\|{\bm{v}}\|_{1,{{\widetilde{w}_E}}}=1} \left|<{\bm{\widetilde{F}}}_h({\bm{u}_h}),{\bm{v}}>{|_{\widetilde{w}_E}}\right|+{h_{T}}|_{\widetilde{w}_E}|| \nabla\cdot(p_h\nabla \psi_h)+f_{T,1} ||_{L^2(\widetilde{w}_E)} 			\nonumber\\
\lesssim &  \|{\bm{\widetilde{F}}}({\bm{u}_h})\|{_{({\bm{\widetilde{Y}}}_h|_{\widetilde{w}_E})^*}}, 
\end{align}
where ${h_{T}}|_{\widetilde{w}_E}$ represents the maximal diameter of $T$ in ${\widetilde{w}_E}$. Inequalities \eqref{paopao2} and \eqref{paopao5} are applied for showing the first and second inequality, respectively, and {\eqref{paopao4}} is used to have the fourth inequality here.

%

Similarly, we derive
\begin{align}
	&h_E^{{1}/{2}}||[\nabla n_h\cdot \bm{n}_E]-[n_h\nabla\psi_h\cdot \bm{n}_E] ||_{L^2(E)} \lesssim \|{\bm{\widetilde{F}}}({\bm{u}_h})\|{_{({\bm{\widetilde{Y}}}_h|_{\widetilde{w}_E})^*}}, \\
	 & h_E^{{1}/{2}}|| [\nabla\psi_h\cdot \bm{n}_E] ||_{L^2(E)} \lesssim \|{\bm{\widetilde{F}}}({\bm{u}_h})\|{_{({\bm{\widetilde{Y}}}_h|_{\widetilde{w}_E})^*}}.
\end{align}
Furthermore, with {${\bm{\widetilde{Y}}}_h|_T\subset{(H_0^1(T))^3}$}, we have
\begin{align}
\| {\bm{\widetilde{F}}}({\bm{u}_h})\|{_{({\bm{\widetilde{Y}}}_h|_T)^*}}
\leq &\|{\bm{F}}({\bm{u}_h})\|{_{({\bm{\widetilde{Y}}}_h|_T)^*}}+\|{\bm{F}}({\bm{u}_h})- {\bm{\widetilde{F}}}({\bm{u}_h})\|{_{({\bm{\widetilde{Y}}}_h|_T)^*}} \nonumber\\
\lesssim &   \|{\bm{F}}({\bm{u}_h})\|_{{-\epsilon},T}+  \|{\bm{F}}({\bm{u}_h})- {\bm{\widetilde{F}}}({\bm{u}_h})\|{_{({\bm{\widetilde{Y}}}_h|_T)^*}}  \label{eq:lowerf1} 
\end{align}
and for the second term in inequality \eqref{eq:lowerf1},
\begin{align}
&\|{\bm{F}}({\bm{u}_h})-{\bm{\widetilde{F}}}({\bm{u}_h})\|_{(\widetilde{{\bm Y}}_h|_T)^*}	\nonumber\\
=&\sup\limits_{{\bm{\delta}}\in \widetilde{{\bm Y}}_{h}|_T,\|{\bm{\delta}}\|_{1,T}=1}{ \Big|}\int_{T}[(f_1-f_{T,1})\delta_1
+(f_2-f_{T,2})\delta_2+(f_3-f_{T,3})\delta_3]d{\bm x}	{ \Big|}\nonumber\\
\lesssim &  \sup\limits_{{\bm{\delta}}\in \widetilde{{\bm Y}}_{h}|_T, \|{\bm{\delta}}\|_{1,T}=1} h_{T}(\|f_1-f_{T,1}\|_{L^2(T)}\|\delta_1\|_{1,T}+
\|f_2-f_{T,2}\|_{L^2(T)}\|\delta_2\|_{1,T}	\nonumber\\
&\quad+\|f_3-f_{T,3}\|_{L^2(T)}\|\delta_3\|_{1,T})	\nonumber\\
\lesssim & \epsilon^{1/2}\varepsilon_T.	\label{eq:lowerf}
\end{align}
Here, $\bm{\delta}=(\delta_1, \delta_2, \delta_3)$. To prove the first inequality in \eqref{eq:lowerf}, the following inequalities are used.
\begin{align}
\|\delta_i\|_{L^2(T)}&\lesssim h_T \|\nabla(\delta_i)\|_{L^2(T)}\lesssim h_T \|\delta_i\|_{1,T},\ i=1,2,3.\nonumber
\end{align}
Thus, the inequality \eqref{th3111} is proved. In the similar way, we get the inequality \eqref{th3112}.
\end{proof}
\begin{remark}
Together with the left inequality in \eqref{basic1} as taking the domain to be $T$, the inequalities \eqref{th3111} and \eqref{th3112} lead to the conclusion that ${\eta\lesssim \epsilon^{-3/2}\|\bm{e}\|_{\epsilon,\Omega}+\varepsilon}$. On the other hand, the right inequality in \eqref{basic1} combined with \eqref{fuh_eta} gives $\|\bm{e}\|_{\epsilon,\Omega} \lesssim \eta + \varepsilon$.
As $f_i$ is piecewise $H^s$ $(0<s\le 1)$ over  $\mathcal{T}_h$, the oscillation term $\varepsilon=\mathcal{O}(h^{1+s})$ \cite{Morin2001Data}, i.e., a higher order term of $\|\bm{e}\|_{\epsilon,\Omega}$. Specifically, $\|f_i-f_{T,i}\|_{L^2(T)}\lesssim h\|f_i\|_{1,T}$,  $\forall f_i\in H^1(T)$,
and $\|f_i-f_{T,i}\|_{L^2(T)}=0$ when $f_i$ is a piecewise-constant function on $\mathcal{T}_h$. Thus, we have the desired result at the leading order as the inequality \eqref{eq:eta_e} shows, i.e., $\underline{C}\eta\leq \|{\bm{e}}\|_{\epsilon,\Omega}\leq \overline{C}\eta$,
with constants $\underline{C}$ and $\overline{C}$ independent of $\bm{u}_h$ and $\bm{f}$, and $\overline{C}$ independent of $\epsilon$.
However, more thorough discussions need to be done carefully for the general case with $f_i\in L^2(T)$.
\end{remark}

\section{Numerical experiments}  \label{numerics}
During the numerical procedure, we involve  the well-known cycle of the adaptive method, that is ``{\it SOLVE} $\rightarrow$ {\it ESTIMATE} $\rightarrow$ {\it MARK} $\rightarrow$ {\it REFINE} ".
The details of the four steps are stated as follows.


\textit{SOLVE}: An adaptive two-grid finite element method is applied to solve the nonlinear system. Instead of using ``uniform" coarse meshes \cite{yang}, we consider the ``adaptive" meshes on the $l$-th level as coarse meshes and denote them by $h_l$. At the sub-level, the initial-guessed ionic concentrations $p^{(0)}$ and $n^{(0)}$ are given and used to solve $\psi^{(0)}$ from Poisson's equation. Then, $\psi^{(0)}$ is used to produce $p^{(1)}$ and $n^{(1)}$ at the next level of iteration by solving Nernst-Planck equations.  The system is solved iteratively at the sub-level till the error tolerance $\|\psi^{(k+1)}-\psi^{(k)}\|_{L^2(\Omega)}<10^{-5}$ is satisfied and hence the solutions $p_{h_l}$, $n_{h_l}$ and $\psi_{h_l}$ are obtained at the $l$-th level. Next, the solutions on $l$-th level of ``coarse" grid are used to find the solutions on the $(l+1)$-th level of adaptive meshes (i.e., the ``fine" grid denoted by $h_{l+1}$). More precisely,  $p_{h_l}$ and $n_{h_l}$ are utilized for solving $\psi_{h_{l+1}}$ from Poisson's equation, then $p_{h_{l+1}}$ and $n_{h_{l+1}}$ are obtained from solving Nernst-Planck equations with $\psi=\psi_{h_{l+1}}$. More details and related analysis are referred to the two-grid method in Ref.~\cite{yang} and adaptive two-grid method in Ref.~\cite{li2018analysis}.

\textit{ESTIMATE}: The a posteriori error estimator $\eta$ introduced in Sec. \ref{sec:posterr} is applied to estimate the error.

\textit{MARK}: The maximum marking strategy {\cite{babuska,Dorfler1996A}} is utilized for remarking the grid.

\textit{REFINE}: The {newest} vertex bisection method \cite{Binev} is used to refine the marked grid. Simultaneously, the new grid is created.

We run the above loops to produce adaptive meshes till the criterion for the fixed maximal degree of freedom $N$ is arrived, i.e., { $N=\mathcal{O}(10^5)$}. All examples are in two dimensions and simulations are performed with the finite element software IFEM. The first and second examples with $\epsilon=1$ are aimed to validate the a posteriori error estimation theory and show the application of the adaptive method for geometrical singularities, respectively.
The third example deals with the boundary layer effect with $\epsilon<1$.

\begin{example} \label{ex1}
We take
\begin{equation}
\left\{
\begin{array}{lll}
p =\sin (2\pi x) \sin (2\pi y),\\
n =\sin (3\pi x)\sin (3\pi y) ,\\
\psi =\sin (\pi x)\sin(\pi y),
\end{array}
\right.
\end{equation}
as an exact solution of the PNP equation system \eqref{model2} with $\epsilon=1$ in the L-shaped domain,  $\Omega= [-1,1]\times[-1,1]\setminus[0,1]\times [-1,0]$, with homogeneous boundary conditions. Consequently, functions $f_1,f_2$  and $f_3 $ on the RHS are determined by the exact solution.
\end{example}

The problem is solved by the adaptive algorithm and Fig. \ref{fig:numsln1} shows the result with the degree of freedom $N=951$ and the error $||\bm{e}||_{L^2}\sim 7\times 10^{-2}$. Panels (a-c) show the numerical solutions of $p$, $n$ and $\psi$, and panel (d) shows the corresponding adaptive mesh refinements. The nonuniform grids illustrate the fact that the a posteriori error estimate works here, nevertheless we do not observe obvious adaptivity, which is caused by the fact that the synthetic solution we choose is sufficiently regular, i.e., the adaptive algorithm is not highly required for solving it efficiently.

\begin {figure}[t!]
\centering
\includegraphics[scale=0.6]{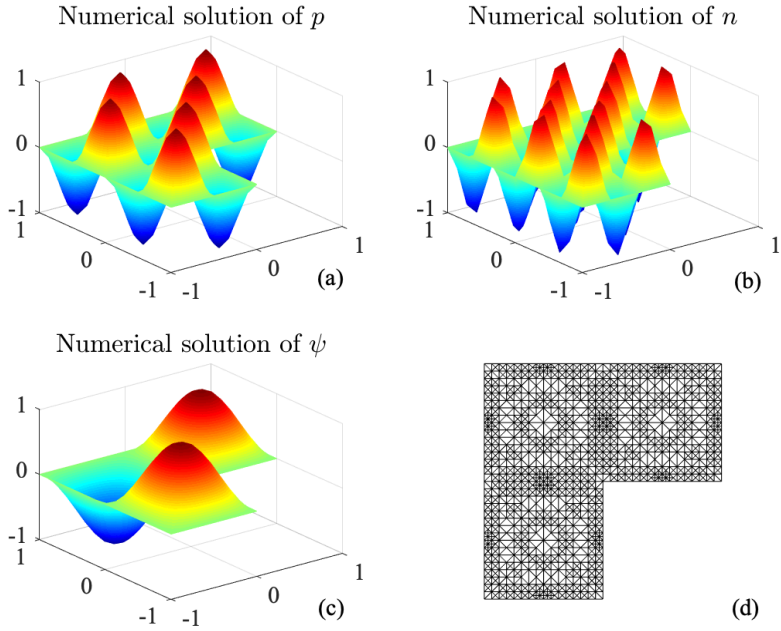}
\caption{The numerical solution of $p$, $n$, $\psi$ \textup{(a,b,c)}, and the final mesh grid \textup{(d)} for Example \ref{ex1} with $N=951$ and $||\bm{e}||_{L^2}\sim 7\times 10^{-2}$.}	\label{fig:numsln1}
\end{figure}

\begin {figure}[h!]
\centering
\includegraphics[scale=0.45]{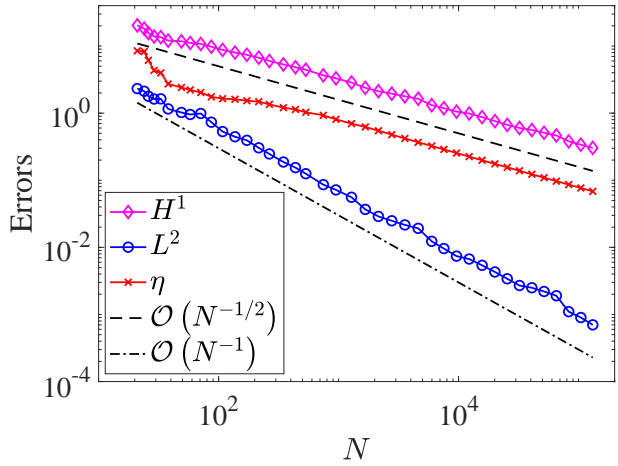}
\caption{The numerical errors are compared with analytical results.  The $H^1$ error of the solution (diamond) and the a posteriori error estimator $\eta$ (circle) both converge to $\mathcal{O}(N^{-1/2})$ for large $N$. The $L^2$ error of the solution (cross) converges to $\mathcal{O}(N^{-1})$.
}
\label{fig:err1}
\end{figure}

Furthermore, we run the algorithm till $N=2\times 10^5$ and test the {a} posteriori error estimation theory by the exact solution. With $N$ increased adaptively, the error analysis is shown in Fig. \ref{fig:err1}  where both $H^1$ and $L^2$ errors of the solution and the a posteriori error estimator $\eta$ are presented. For $N>10^2$, the $H^1$ and $L^2$ errors converge to $\mathcal{O}(N^{-1/2})$ and $\mathcal{O}(N^{-1})$, respectively, which are as expected for linear finite element interpolations. For large number of $N$, i.e., $N>10^3$, $\eta$ converges to $\mathcal{O}(N^{-1/2})$,  embodying that the error of the numerical solution is well controlled by the a posteriori error estimator $\eta$ of which the reliability and efficiency are hence numerically shown.

\begin{example} \label{ex2}
In order to demonstrate the adaptive performance of the estimator $\eta$ for geometrical singularities, we choose $f_1=f_2=f_3=1$ and rewrite the problem~\eqref{model2} with $\epsilon=1$ as follows,
\begin{equation}
\left\{
\begin{array}{lll}
-\Delta p-\nabla\cdot(p\nabla\psi)=1, \\
-\Delta n+\nabla\cdot(n\nabla\psi)=1, \\
-\Delta \psi-p+n=1,
\end{array}
\right.
\end{equation}
by which we need find $p(\bm{x}),n(\bm{x}),\psi(\bm{x}) \in H_0^1({\Omega})$ with the same computational domain as in Example \ref{ex1}.
\end{example}

\begin {figure}[t!]
\centering
\includegraphics[scale=0.55]{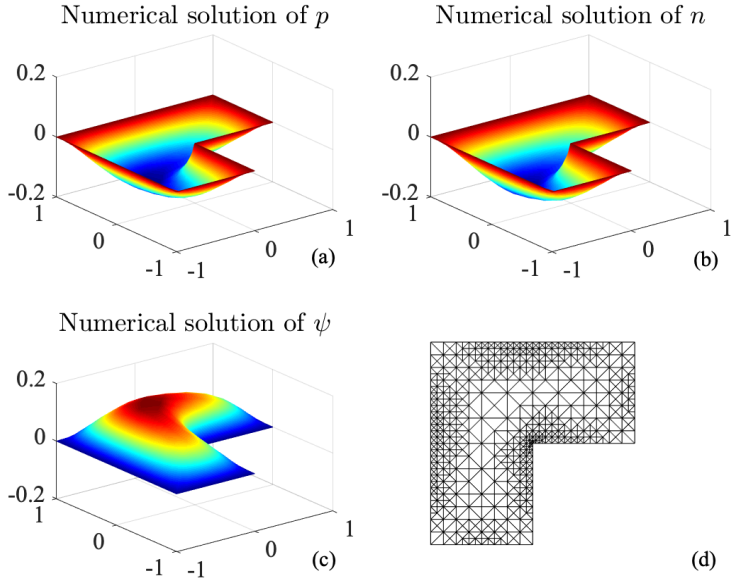}
\caption{The numerical solution of $p$, $n$ and $\psi$  \textup{(a,b,c)} and the mesh grid \textup{(d)} of Example \ref{ex2} for $N=1243$ and $\eta\approx 7\times 10^{-3}$.}
\label{fig:numsol2}
\end{figure}

To illustrate the adaptive mesh refinements clearly, we report the mesh grid for $N=1243$ in Fig. \ref{fig:numsol2} (panel (d)). The corresponding numerical results for $p$, $n$ and $\psi$ presented in panels (a,b,c) are quite smooth, however, it is evident to observe the adaptive mesh refinements near the corner point $(0, 0)$, which indicates the well performance of the a posteriori error estimator $\eta$ proposed in current work. More specifically, we calculate the numerical value of the indicator $\eta$ versus $N$ and present it in Fig. \ref{fig:err2}. No $H^1$ or $L^2$ errors are shown because we  have no analytical solution here. As expected, the a posteriori error estimator converges to the optimal order from the related theory, i.e., $\mathcal{O}(N^{-1/2})$.

\begin {figure}[h!]
\centering
\includegraphics[scale=0.55]{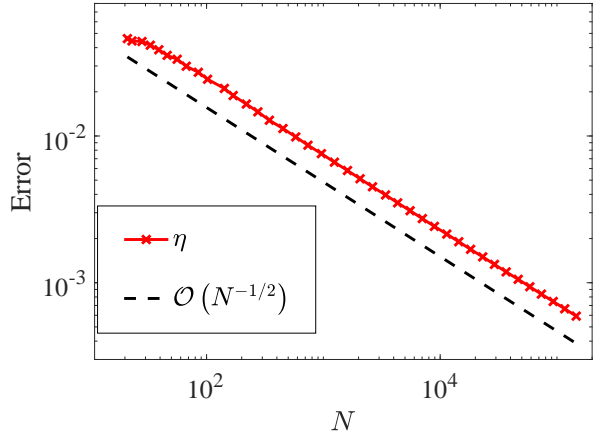}
\caption{The a posteriori error estimator $\eta$ for Example \ref{ex2} which converges to $\mathcal{O}\left(N^{-1/2}\right)$.
}
\label{fig:err2}
\end{figure}

{
\begin{example} \label{ex3}
In this example, we consider the effect of boundary layer or the thin Debye layer thickness, i.e., the general steady-state PNP equations \eqref{model2} as follows:
to find $p(\bm{x}),n(\bm{x}),\psi(\bm{x}) \in H_0^1({\Omega})$, such that
\begin{equation}\label{varepsilon}
\begin{cases}
\begin{array}{lll}
-\Delta p-\nabla\cdot(p\nabla\psi)=f_1, \\
-\Delta n+\nabla\cdot(n\nabla\psi)=f_2, \\
-\epsilon\Delta \psi-p+n=f_3,
\end{array}
\end{cases}
\end{equation}
with $0<\epsilon<1$, and $\Omega=[0,1]\times[0,1]$. We take
\begin{equation}\label{varepsilon1}
\begin{cases}
\begin{array}{lll}
p=e^{-x/\sqrt{\epsilon}}+e^{-y/\sqrt{\epsilon}},\\
n=e^{-2x/\sqrt{\epsilon}}+e^{-2y/\sqrt{\epsilon}}, \\
\psi=e^{-3x/\sqrt{\epsilon}}+e^{-3y/\sqrt{\epsilon}},
\end{array}
\end{cases}
\end{equation}
as an exact solution of Eq. \eqref{varepsilon}. Similarly to Example \ref{ex1}, the functions $f_1$, $f_2$ and $f_3$ on the RHS are determined by the exact solution. The boundary conditions on $\partial\Omega$ are then chosen to satisfy \eqref{varepsilon1}.
\end{example}

For a small value of $\epsilon$,  the exact solution \eqref{varepsilon1} predicts boundary layer effects near $x=0$ and $y=0$. As a comparison, we run the adaptive algorithm for $\epsilon=0.1$ and $0.01$ and show the results in Fig. \ref{fig:eps}.
In both cases, the consistent convergence of the $\epsilon$-dependent error and $\eta$ is observed as increasing $N$, i.e, $\mathcal{O}(N^{-1/2})$ (at the top). The corresponding mesh grids for typical degrees of freedom are presented at the bottom. The adaptivity can be demonstrated as more condensed grids near the boundaries for $\epsilon=0.01$ are visualized as compared to that of $\epsilon=0.1$.

}

\begin {figure}[h!]\label{fig:eps}
\centering
\includegraphics[scale=0.55]{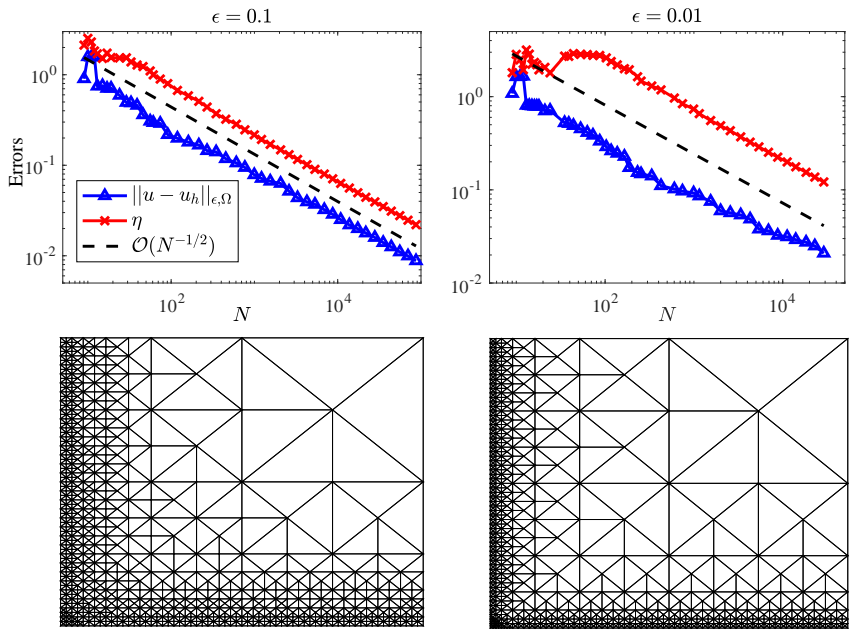}
\caption{The errors (top) and corresponding mesh grids (bottom) for different $\epsilon$. The degrees of freedom in the mesh grids are $N=606$ and $602$ for $\epsilon=0.1$ and $0.01$, respectively.}
\end{figure}

\section{Conclusions} \label{conclusions}
In this paper, {the residual-based a posteriori error estimator} has been adopted for the adaptive analysis of steady-state PNP equations where the nonlinearity and strong coupling are focused. During the theoretical study of the {a} posteriori error estimation, we have established the relationship between the a posteriori error estimator and the error of solution with the help of strategies from Ref. \cite{Verfuhrt},
so as to demonstrate the efficiency and reliability  of the error estimator.
Additionally, by taking $G$-derivatives of the nonlinear PNP equations, this paper has constructed a corresponding linear problem of which the solution existence and uniqueness have been proved.

We have successfully shown the rationality of theoretical conclusions by numerical results. The efficiency and reliability of the a posteriori error estimator are confirmed numerically in Example \ref{ex1}. The adaptive performances are given in Example \ref{ex2} and \ref{ex3} with  treating geometrical  singularities and boundary layer effects, respectively. Nevertheless, in view of more thorough investigations that have not been done here, we only consider this work as a very starting point of adaptive methods for the PNP system.
On the numerical side, to the authors' best knowledge, the convergence and stability of the entire adaptive algorithm for steady-state PNP has not been studied systematically, despite some of the existed methods \cite{yang,li2018analysis}. At last, the general estimation analysis  and the adaptive method for time-dependent PNP are of many more interests. Notably, the general adaptive method that includes the temporal adaptivity is more challenging and will be saved as our future goal.

%


\appendix
\section{Proof of solution existence and uniqueness}\label{app:existanduniq}
\begin{lemma} \label{lemma:poisson}
\cite{CIARLET199117} There exists a unique weak solution $u\in H_0^1(\Omega)$ of
\begin{equation}
	-\Delta u=f,\ in\ \Omega,\ u|_{\partial\Omega}=0,
\end{equation}
$\forall f\in L^2(\Omega)$, and $|u|_{1,\Omega}\leq C_p\|f\|_{L^2(\Omega)}$, where $C_p$ is a constant that depends on $\Omega$ only.
\end{lemma}



\begin{lemma}(Schauder Fixed Point Theorem \cite{Zeidler})
For a Banach space X and a bounded closed convex subset $K\subset X$, if T is a fully continuous operator from K to itself, there exists $\bm{x}\in K$ satisfying $T\bm{x}=\bm{x}$.
\end{lemma}
According to Green formula and the definition of $G$-derivative, Eq. \eqref{modelDF} is rewritten as follows:
{$\forall {\bm{R}}=(F_1,F_2,F_3)\in (L^2(\Omega))^3$, to find ${\bm{\phi}}=(\phi_1,\phi_2,\phi_3)\in (H_0^1(\Omega))^3$, such that,
\begin{equation}\label{Q}
\left\{
\begin{array}{lll}
-\nabla\cdot(\nabla \phi_1+p\nabla\phi_3+\phi_1\nabla \psi)=F_1, \\
-\nabla\cdot( \nabla \phi_2-n\nabla\phi_3-\phi_2\nabla \psi)=F_2, \\
-\epsilon \Delta \phi_3-\phi_1+\phi_2=F_3.
\end{array}
\right.
\end{equation}

\underline{\it Solution Existence:}	\label{app:exist}

We define a bounded closed convex subset of $L^2(\Omega)\times  L^2(\Omega)$ as,
$$
K=\{(\widetilde{f}_1,\widetilde{f}_2) |\ \|\nabla \widetilde{f}_i\|_{L^2(\Omega)}\leq \beta_0 \,\ \textup{and}\ \widetilde{f}_i|_{\partial \Omega}=0,\ i=1,2 \},
$$
with the constant $\beta_0$ to be determined later. Next, we aim to find a mapping $M$ from $K$ into itself.

It's  straightforward to have that $\forall (\widetilde{\phi}_1,\widetilde{\phi}_2)\in K$ and $\forall F_3\in L^2(\Omega)$, $\widetilde{\phi}_3=\widetilde{\phi}_3(\widetilde{\phi}_1,\widetilde{\phi}_2)$ is a unique solution in $H_0^1(\Omega)$ of the following equation,
\begin{equation}\label{1cvlinearquestion}
< \epsilon\nabla \widetilde{\phi}_3,\nabla v_3>-<\widetilde{\phi}_1-\widetilde{\phi}_2,v_3>=< F_3,v_3> ,\quad \forall v_3\in H_0^1(\Omega).
 \end{equation}
The decoupled variational boundary-value problems \eqref{2cvlinearquestion} for $\phi_1$, $\phi_2\in H_0^1(\Omega)$ are considered now,
\begin{equation}\label{2cvlinearquestion}
\left\{
\begin{array}{ll}
<\nabla \phi_1+p\nabla\widetilde{\phi}_3+\widetilde{\phi}_1\nabla\psi,\nabla v_1> = <F_1,v_1>,\quad\forall v_1\in H_0^1(\Omega), \\
<\nabla \phi_2-n\nabla\widetilde{\phi}_3-\widetilde{\phi}_2\nabla\psi,\nabla v_2> = < F_2, v_2>,\quad\forall v_2\in H_0^1(\Omega),
\end{array}
\right.
\end{equation}
where $F_1,\ F_2\in L^2(\Omega)$.
Owing to the fact that $p,\ n\in H^1(\Omega)$ and $-\epsilon\Delta\psi=p-n+f_3$, we have $ \nabla\cdot(p\nabla\widetilde{\phi}_3+\widetilde{\phi}_1\nabla\psi)\in L^2(\Omega)$ and $ \nabla\cdot(n\nabla\widetilde{\phi}_3+\widetilde{\phi}_2\nabla\psi)\in L^2(\Omega)$. Referring to the solution existence and uniqueness of Eqs. \eqref{2cvlinearquestion}, we express the mapping $M$ as
\begin{equation}
(\phi_1,\phi_2)=M(\widetilde{\phi}_1,\widetilde{\phi}_2).
\end{equation}

After taking $v_1=\phi_1$ in the first equation of \eqref{2cvlinearquestion}, we have
\begin{align}
&<\nabla \phi_1,\nabla\phi_1>\nonumber\\
=&<F_1,\phi_1>-<p\nabla\widetilde{\phi}_3+\widetilde{\phi}_1\nabla\psi,\nabla\phi_1>\nonumber\\
\leq& \|F_1\|_{L^2(\Omega)}\|\phi_1\|_{L^2(\Omega)}+\|p\|_{L^{\infty}(\Omega)}\|\nabla\widetilde{\phi}_3\|_{L^2(\Omega)}\|\nabla \phi_1\|_{L^2(\Omega)} \nonumber\\
&+\sqrt{2}\|\nabla\psi\|_{L^{\infty}(\Omega)}\|\widetilde{\phi}_1\|_{{L^2(\Omega)}}\|\nabla\phi_1\|_{{L^2(\Omega)}}\nonumber\\
\leq& C_p\|F_1\|_{L^2(\Omega)}\|\nabla\phi_1\|_{L^2(\Omega)}+\|p\|_{L^{\infty}(\Omega)}\|\nabla\widetilde{\phi}_3\|_{L^2(\Omega)}\|\nabla \phi_1\|_{L^2(\Omega)}	\nonumber\\
&+\sqrt{2}C_p\|\nabla\psi\|_{L^{\infty}(\Omega)}\|\nabla\widetilde{\phi}_1\|_{{L^2(\Omega)}}\|\nabla\phi_1\|_{{L^2(\Omega)}}\nonumber\\
\leq& C_p(\|F_1\|_{L^2(\Omega)}\|\nabla\phi_1\|_{L^2(\Omega)}+{\|p\|_{L^{\infty}(\Omega)}\epsilon^{-1}}\|F_3+\widetilde{\phi}_1-\widetilde{\phi}_2\|_{L^2(\Omega)}\|\nabla \phi_1\|_{L^2(\Omega)}	\nonumber\\
&+\sqrt{2}\|\nabla\psi\|_{L^{\infty}(\Omega)}\|\nabla\widetilde{\phi}_1\|_{{L^2(\Omega)}}\|\nabla\phi_1\|_{{L^2(\Omega)}})\nonumber\\
\leq& C_p\left(\|F_1\|_{L^2(\Omega)}+\|p\|_{L^{\infty}(\Omega)}\epsilon^{-1}\|F_3\|_{L^2(\Omega)}\right)\|\nabla\phi_1\|_{L^2(\Omega)}+ \left[C_p^2\epsilon^{-1}\|p\|_{L^{\infty}(\Omega)}(\|\nabla\widetilde{\phi}_1\|_{{L^2(\Omega)}} \right.\nonumber\\
&\left.+\|\nabla\widetilde{\phi}_2\|_{{L^2(\Omega)}})	
+\sqrt{2}C_p\|\nabla\psi\|_{L^{\infty}(\Omega)}\|\nabla\widetilde{\phi}_1\|_{{L^2(\Omega)}}  \right]\|\nabla\phi_1\|_{L^2(\Omega)}, \label{ineq:phi1}
\end{align}
where in the second and fourth inequalities the Poincar\'{e} inequality is applied with the constant $C_p$ the same as in Lemma \ref{lemma:poisson}, and Eq. \eqref{1cvlinearquestion} is used in the third inequality.

Recall that $\|\nabla \widetilde{\phi}_1\|_{L^2(\Omega)}\leq \beta_0$ and $\|\nabla \widetilde{\phi}_2\|_{L^2(\Omega)}\leq \beta_0$, we restrict $\|\nabla \phi_1\|_{L^2(\Omega)}\leq\beta_0$ by using inequality \eqref{ineq:phi1}, i.e.,
\begin{align}
&\|\nabla \phi_1\|_{L^2(\Omega)} \nonumber\\
\leq& C_p\Big(\|F_1\|_{L^2(\Omega)}+\|p\|_{L^{\infty}(\Omega)}\epsilon^{-1}\|F_3\|_{L^2(\Omega)}\Big)
+\left(2C_p^2\epsilon^{-1}\|p\|_{L^{\infty}(\Omega)}+\sqrt{2}C_p\|\nabla\psi\|_{L^{\infty}(\Omega)}\right)\beta_0\nonumber\\
\leq& \beta_0,
\end{align}
which leads to
\begin{equation}\label{beta}
\beta_0>\frac{C_p\left(\|F_1\|_{L^2(\Omega)}+\epsilon^{-1}\|p\|_{L^{\infty}(\Omega)}\|F_3\|_{L^2(\Omega)}\right)}{1-2C_p^2\epsilon^{-1}\|p\|_{L^{\infty}(\Omega)}
-\sqrt{2}C_p\|\nabla\psi\|_{L^{\infty}(\Omega)}} = \beta_1,
\end{equation}
with requiring
\begin{equation}
1-2C_p^2\|p\|_{L^{\infty}(\Omega)}\epsilon^{-1}
-\sqrt{2}C_p\|\nabla\psi\|_{L^{\infty}(\Omega)}>0.
\end{equation}
Similarly, $\|\nabla \phi_2\|_{L^2(\Omega)}\leq \beta_0$ gives
\begin{equation}
	\beta_0>\frac{C_p\left(\|F_2\|_{L^2(\Omega)}+\epsilon^{-1}\|n\|_{L^{\infty}(\Omega)}\|F_3\|_{L^2(\Omega)}\right)}{1-2C_p^2\epsilon^{-1}\|n\|_{L^{\infty}(\Omega)}
-\sqrt{2}C_p\|\nabla\psi\|_{L^{\infty}(\Omega)}}= \beta_2,
\end{equation}
with  requiring
\begin{equation}1-2C_p^2\epsilon^{-1}\|n\|_{L^{\infty}(\Omega)}
-\sqrt{2}C_p\|\nabla\psi\|_{L^{\infty}(\Omega)}>0.
\end{equation}

Thus, when
\begin{equation} \label{condexist}
	\|\nabla\psi\|_{L^{\infty}(\Omega)}+\sqrt{2}C_p{\epsilon^{-1}}\max\{\|p\|_{L^{\infty}(\Omega)}, \|n\|_{L^{\infty}(\Omega)}  \}<1 \left /\left(\sqrt{2}C_p\right)\right.,
\end{equation}
the constant $\beta_0>\max\{\beta_1,\beta_2\}$ can ensure that $K$ is a bounded closed convex subset of $H_0^1(\Omega)\times  H_0^1(\Omega)$ and $M$ is a mapping from $K$ to itself, i.e., $M(K)\subset K$.

As a next step, we prove that the mapping $M$ is compact and  continuous, that is, a fully continuous operator.

\textit{Compactness}. From Eq. \eqref{2cvlinearquestion} where $F_1+\nabla\cdot(p\nabla\widetilde{\phi}_3+\widetilde{\phi}_1\nabla\psi )\in L^2(\Omega),$ we have,  $\forall \ (\widetilde{\phi}_1,\widetilde{\phi}_2)\in K$, $\exists\ C_1>0$, s.t., $\|F_1+\nabla\cdot(p\nabla\widetilde{\phi}_3+\widetilde{\phi}_1\nabla\psi)\|_{{L^2(\Omega)}}\leq C_1 $, and hence
\begin{align}\label{5}
<\nabla \phi_1,\nabla\phi_1>
=&<F_1 + \nabla\cdot(p\nabla\widetilde{\phi}_3+\widetilde{\phi}_1\nabla\psi),\phi_1>\nonumber\\
\leq& \|F_1 + \nabla\cdot(p\nabla\widetilde{\phi}_3+\widetilde{\phi}_1\nabla\psi)\|_{{L^2(\Omega)}}\|\phi_1\|_{{L^2(\Omega)}} \nonumber\\
\leq& C_1C_p|\phi_1|_{1,\Omega}, 
\end{align}
i.e., $|\phi_1|_{1,\Omega}\leq C_1C_p$, and similarly, $|\phi_2|_{1,\Omega}\leq C_2C_p$, where $C_1$, $C_2$ and $C_p$ are independent of $\phi_1$ and $\phi_2$.
Consequently, $M$ maps a bounded set in  $H_0^1(\Omega)\times  H_0^1(\Omega)$ into a bounded set in
 $H_0^1(\Omega)\times  H_0^1(\Omega)$ which is compactly embedded in ${{L^2(\Omega)}}\times {{L^2(\Omega)}}$.
Thus, $M$ is a compact operator.

\textit{Continuity}. Taking $(\widetilde{\phi}_1,\widetilde{\phi}_2)=(\widetilde{\phi}'_1,\widetilde{\phi}'_2)$ and
$(\widetilde{\phi}''_1, \widetilde{\phi}''_2)\in H_0^1(\Omega)\times  H_0^1(\Omega)$ in Eq. \eqref{1cvlinearquestion},  we have $\widetilde{\phi}'_3 = \widetilde{\phi}'_3(\widetilde{\phi}'_1,\widetilde{\phi}'_2)$ and $\widetilde{\phi}''_3=\widetilde{\phi}''_3(\widetilde{\phi}''_1,\widetilde{\phi}''_2)$, respectively. And by Eq. \eqref{2cvlinearquestion}, we have
{$(\phi'_1,\phi'_2)=M(\widetilde{\phi}'_1,\widetilde{\phi}'_2)$ and $(\phi_1'',\phi_2'')=M(\widetilde{\phi}''_1,\widetilde{\phi}''_2)$},
then,
\begin{align}\label{9}
&<\nabla (\phi'_1-\phi_1''),\nabla (\phi_1'-\phi_1'')>\nonumber\\
=&-<p\nabla(\widetilde{\phi}_3''-\widetilde{\phi}'_3)+\nabla\psi(\widetilde{\phi}_1''-\widetilde{\phi}'_1),\nabla (\phi'_1-\phi_1'')> \nonumber\\
\leq& \|p\nabla(\widetilde{\phi}_3''-\widetilde{\phi}'_3)+\nabla\psi(\widetilde{\phi}_1''-\widetilde{\phi}_1')\|_{{{L^2(\Omega)}}}\|\nabla (\phi_1'-\phi_1'')\|_{{{L^2(\Omega)}}}. 
\end{align}
Furthermore, from \eqref{1cvlinearquestion}  we get
\begin{equation}
\left\{
\begin{array}{lll}
&-{\epsilon}\Delta (\widetilde{\phi}_3'-{\widetilde{\phi}_3}'')=(\widetilde{\phi}_1'-{\widetilde{\phi}_1}'')-(\widetilde{\phi}_2'-{\widetilde{\phi}_2}''),\ \textup{in}\ \Omega ,  \\
&\widetilde{\phi}_3' - {\widetilde{\phi}_3}''=0,\ \textup{on}\ \partial\Omega .
\end{array}
\right.
\end{equation}
 Now, in the RHS of inequality \eqref{9},
\begin{align}\label{8}
&\int_{\Omega}(p\nabla(\widetilde{\phi}_3''-\widetilde{\phi}_3')+\nabla\psi(\widetilde{\phi}_1''-\widetilde{\phi}_1'))^2d\bm{x}\nonumber	\\
\leq& 2 \max\{\|p\|_{L^{\infty}(\Omega)}^2,\|\psi\|_{W^{1,\infty}}^2\}\int_{\Omega}[(\nabla (\widetilde{\phi}_3''-\widetilde{\phi}_3'))^2+(\widetilde{\phi}_1''-\widetilde{\phi}_1')^2] d\bm{x}	\nonumber\\
\leq& 2\max\{\|p\|_{L^{\infty}(\Omega)}^2,\|\psi\|_{W^{1,\infty}}^2\}({2C_p^2\epsilon^{-1}}+1)(\|\widetilde{\phi}_1'-\widetilde{\phi}_1''\|_{L^2(\Omega)}^2
+\|\widetilde{\phi}_2'-\widetilde{\phi}_2''\|_{{{L^2(\Omega)}}}^2),	
\end{align}
{where in the second inequality Lemma \ref{lemma:poisson} is applied.}
With \eqref{8} plugged into \eqref{9}, we have,
\begin{equation}\label{9-app}
|\phi_1'-\phi_1''|_{1,\Omega}\lesssim \|\widetilde{\phi}_1'-\widetilde{\phi}_1''\|_{L^2(\Omega)}+\|\widetilde{\phi}_2'-\widetilde{\phi}_2''\|_{{{L^2(\Omega)}}},
\end{equation}
and analogously,
\begin{equation}\label{10}
|\phi_2'-\phi_2''|_{1,\Omega}\lesssim\|\widetilde{\phi}_2'-\widetilde{\phi}''_2\|_{L^2(\Omega)}+\|\widetilde{\phi}_1'-\widetilde{\phi}''_1\|_{{{L^2(\Omega)}}}.
\end{equation}
Thus,
\begin{equation}\label{10-2}
\|\phi_1'-\phi_1''\|_{1,\Omega}+\|\phi_2'-\phi_2''\|_{1,\Omega}\leq
C\left(\|\widetilde{\phi}_2'-\widetilde{\phi}_2''\|_{1,\Omega}+\|\widetilde{\phi}_1'-\widetilde{\phi}_1''\|_{1,\Omega}\right),
\end{equation}
with the constant $C$ independent of $\phi_i'$, $\phi_i''$, $\widetilde{\phi}_i'$ and $\widetilde{\phi}_i''$ ($i=1,2$), i.e., $M$ is continuous.

In summary, we can identify a bounded  convex closed set $K\subset H_0^1(\Omega)\times  H_0^1(\Omega)\subset L^2(\Omega)\times L^2(\Omega)$, with a fully continuous mapping  $M$ of $K$ to itself, {i.e., $M(K)\subset K$}. The Schauder fixed-point theorem is then applied to show that there exists $(\widehat{\phi}_1,\widehat{\phi}_2)\in K$ satisfying
\begin{equation}
	M(\widehat{\phi}_1,\widehat{\phi}_2)=(\widehat{\phi}_1,\widehat{\phi}_2).
\end{equation}
Recall in \eqref{Q}, $\phi_3$ can be solved uniquely with taking $(\phi_1,\phi_2)=(\widehat{\phi}_1,\widehat{\phi}_2)\in K\subset L^2(\Omega)\times L^2(\Omega)$, i.e., $\phi_3=\widehat{\phi}_3(\widehat{\phi}_1,\widehat{\phi}_2)\in H_0^1(\Omega)$.
Therefore, $(\widehat{\phi}_1,\widehat{\phi}_2,\widehat{\phi}_3)$ is a solution of
\eqref{modelDF}.

\underline{\it Solution Uniqueness:}	\label{app:unique}

Assuming both ${\bm{\phi}}'=(\phi_1',\phi_2',\phi_3')$ and ${\bm{\phi}}''=(\phi_1'',\phi_2'',\phi_3'')$ satisfy \eqref{Q}, we prove $\phi_1'\equiv\phi_1''$ and $\phi_2'\equiv\phi_2''$ as following, and thus $\phi_3'\equiv\phi_3''$.

Let $\widetilde{\bm{v}}=(v_1,v_2,0)$, $\forall\  (v_1,v_2)\in H_0^1(\Omega)\times H_0^1(\Omega)$, if $\bm{\phi} = (\phi_1, \phi_2, \phi_3)$ is a solution of Eq.\eqref{Q}, we have
\begin{align}
	&<\bm{DF}(\bm{u})\bm{\phi}-\bm{R},\widetilde{\bm{v}}>\nonumber \\
=&<\nabla \phi_1+\phi_1\nabla\psi,\nabla v_1>-<\nabla\cdot(p\nabla \phi_3),v_1>-<F_1,v_1>  \nonumber\\
&+<\nabla \phi_2-\phi_2\nabla\psi,\nabla v_2>+<\nabla\cdot(n\nabla \phi_3),v_2>-<F_2,v_2>\nonumber \\
= & 0, 
\end{align}
and
\begin{equation}
	<\bm{DF}(\bm{u})\bm{\phi}'-\bm{R},\widetilde{\bm{v}}> = <\bm{DF}(\bm{u})\bm{\phi}''-\bm{R},\widetilde{\bm{v}}> =0.
\end{equation}

We define $\Phi_1=\phi_1'-\phi_1'',\ \Phi_2=\phi_2'-\phi_2'',\ \Phi_3=\phi_3'-\phi_3''$, and set $\widetilde{\bm{v}}= \bm{\Phi} = (\Phi_1,\Phi_2,0)$, then have,
\begin{align}
	&<\bm{DF}(\bm{u})\bm{\phi}'-\bm{R}, \bm{\Phi}>-<\bm{DF}(\bm{u})\bm{\phi}''-\bm{R}, \bm{\Phi}> \nonumber\\
=&<\nabla \Phi_1,\nabla \Phi_1>+<\nabla \Phi_2,\nabla \Phi_2>+<\Phi_1\nabla\psi,\nabla \Phi_1>+<-\Phi_2\nabla\psi,\nabla \Phi_2>	\nonumber\\
&+<\nabla \Phi_3,p\nabla\Phi_1 - n\nabla\Phi_2>.	\label{exists4}
\end{align}
The last three terms in \eqref{exists4} are considered first, which are
\begin{align}
&|<\Phi_1\nabla\psi,\nabla \Phi_1>+<-\Phi_2\nabla\psi,\nabla \Phi_2>|\nonumber\\
\leq& \sqrt{2}\|\nabla\psi\|_{L^{\infty}(\Omega)}\left(\|\Phi_1\|_{L^2(\Omega)}\|\nabla\Phi_1\|_{L^2(\Omega)}
+\|\Phi_2\|_{L^2(\Omega)}\|\nabla\Phi_2\|_{L^2(\Omega)} \right)	\nonumber\\
\leq& \frac{\sqrt{2}}{2}\|\nabla\psi\|_{L^{\infty}(\Omega)}\left(\|\Phi_1\|_{1,\Omega}^2+\|\Phi_2\|_{1,\Omega}^2\right),	\label{ineq:3terms1}
\end{align}
and,
\begin{align}
&|<\nabla \Phi_3, p\nabla\Phi_1 - n\nabla\Phi_2>|	\nonumber\\
\leq& \|\nabla \Phi_3\|_{{{L^2(\Omega)}}}\|p\nabla\Phi_1 - n\nabla\Phi_2\|_{{{L^2(\Omega)}}}	\nonumber\\
\leq&{ C_p\epsilon^{-1}}\|\Phi_1-\Phi_2\|_{{{L^2(\Omega)}}}\cdot \max\{ \|p\|_{L^{\infty}(\Omega)}, \|n\|_{L^{\infty}(\Omega)}\}{\|}|\nabla\Phi_1|+ |\nabla\Phi_2|{\|}_{{{L^2(\Omega)}}}	\nonumber\\
\leq& \frac{1}{2}C_p \epsilon^{-1}\max\{ \|p\|_{L^{\infty}(\Omega)},\|n\|_{L^{\infty}(\Omega)}\}(\|\Phi_1-\Phi_2\|_{{{L^2(\Omega)}}}^2+{\|}|\nabla\Phi_1|+|\nabla\Phi_2|{\|}_{{{L^2(\Omega)}}}^2)	\nonumber\\
\leq& C_p\epsilon^{-1}\max\{ \|p\|_{L^{\infty}(\Omega)},\|n\|_{L^{\infty}(\Omega)}\}\int_{\Omega}(|\Phi_1|^2+|\Phi_2|^2+|\nabla\Phi_1|^2+|\nabla\Phi_2|^2)d\bm{x}	\nonumber\\
=& C_p \epsilon^{-1}\max\{ \|p\|_{L^{\infty}(\Omega)},\|n\|_{L^{\infty}(\Omega)}\}(\|\Phi_1\|_{1,\Omega}^2+\|\Phi_2\|_{1,\Omega}^2),	\label{ineq:3terms2}
\end{align}
where in the second inequality Lemma \ref{lemma:poisson} is applied.

For the first and second terms on the RHS of \eqref{exists4},   the Poincar\'{e} inequality gives $\|\Phi_i\|_{1,\Omega}^2\leq (1+C_p^2)\|\nabla\Phi_i\|_{L^2(\Omega)}^2,\ i=1,2$, together with  \eqref{ineq:3terms1} and \eqref{ineq:3terms2} we have
\begin{align}
&|<\bm{DF}(\bm{u})\bm{\phi}'-\bm{R}, \bm{\Phi}>-<\bm{DF}(\bm{u})\bm{\phi}''-\bm{R}, \bm{\Phi}>| \nonumber \\
\geq& <\nabla \Phi_1,\nabla \Phi_1>+<\nabla \Phi_2,\nabla \Phi_2>	\nonumber\\
&-|<\Phi_1\nabla\psi,\nabla \Phi_1>+<-\Phi_2\nabla\psi,\nabla \Phi_2>|-|<\nabla \Phi_3, p\nabla\Phi_1 - n\nabla\Phi_2>|	\nonumber\\
\geq& \left[{1}/{(1+C_p^2)} - \left({\sqrt{2}}/{2}\|\nabla\psi\|_{L^{\infty}(\Omega)}+{C_p\epsilon^{-1}} \max\{ \|p\|_{L^{\infty}(\Omega)},\|n\|_{L^{\infty}(\Omega)}\}\right)\right] \nonumber\\
&\cdot (\|\Phi_1\|_{1,\Omega}^2+\|\Phi_2\|_{1,\Omega}^2).	
\end{align}
%
Obviously, when ${1}/{(1+C_p^2)}-({\sqrt{2}}/{2}\|\nabla\psi\|_{L^{\infty}(\Omega)}+C_p \epsilon^{-1}\max\{ \|p\|_{L^{\infty}(\Omega)},\|n\|_{L^{\infty}(\Omega)}\})>0$,
we have
$\|\Phi_1\|_{1,\Omega}^2+\|\Phi_2\|_{1,\Omega}^2\leq 0$,
that is $\Phi_1\equiv\Phi_2\equiv0$, i.e., $\phi_1'\equiv\phi_1''$ and $\phi_2'\equiv\phi_2''$.
Finally, the solution existence and uniqueness of Poisson's equation, $-{\epsilon}\Delta \Phi_3=0$ in $ \Omega$ with $\Phi_3=0$ on $\partial\Omega$, lead to $\Phi_3\equiv 0$, i.e., $\phi_3'=\phi_3''$.


Therefore, we prove $\bm{\phi}'\equiv\bm{\phi}''$ with
\begin{equation} \label{condunique}
	 \|\nabla\psi\|_{L^{\infty}(\Omega)}+\sqrt{2}C_p{\epsilon^{-1} }\max\{ \|p\|_{L^{\infty}(\Omega)},\|n\|_{L^{\infty}(\Omega)}\}< \sqrt{2}\left/\left(1+C_p^2\right)\right.\leq 1\left/\left(\sqrt{2}C_p\right)\right..
\end{equation}

\begin{remark}
	We mention that \eqref{condexist} and \eqref{condunique} here are sufficient conditions for solution existence and uniqueness, respectively. More rigorous analysis is needed for deriving the necessary conditions, which, however, is neglected in this work.
\end{remark}

\begin{acknowledgements}
T. Hao and X. Xu acknowledge the financial support from NSFC (No. {11671302}). M. Ma acknowledges the financial support from NSFC (No. 11701428), ``Chen Guang'' project supported by Shanghai Municipal Education Commission and Shanghai Education Development Foundation, and the Fundamental Research Funds for the Central Universities.
\end{acknowledgements}


%
 \section*{Conflict of interest}

 The authors declare that they have no conflict of interest.

\section*{Data availability statement}
The datasets generated during and/or analysed during the current study are available from the corresponding author on reasonable request.


\begin{thebibliography}{}
%
%
\bibitem{ainsworth1999reliable}
{\sc M.~Ainsworth and I.~Babu{\v{s}}ka}, {\em Reliable and robust a posteriori
  error estimation for singularly perturbed reaction-diffusion problems}, SIAM
  J. Num. Anal., 36 (1999), pp.~331--353.

\bibitem{Oden}
{\sc M.~Ainsworth and J.~T. Oden}, {\em A Posteriori Error Estimation in Finite
  Element Analysis}, New York: John Wiley {\&} Sons, 2000.

\bibitem{ainsworth2011fully}
{\sc M.~Ainsworth and T.~Vejchodsk{\'{y}}}, {\em Fully computable robust a
  posteriori error bounds for singularly perturbed reaction-diffusion
  problems}, Numer. Math., 119 (2011), pp.~219--243.

\bibitem{barcilon1997SIAM}
{\sc V.~Barcilon, D.~Chen, R.~S. Eisenberg, and J.~W. Jerome}, {\em Qualitative
  properties of steady-state {Poisson-Nernst-Planck} systems: perturbation and
  simulation study}, SIAM J. Appl. Math., 57 (1997), pp.~631--648.

\bibitem{Binev}
{\sc P.~Binev, W.~Dahmen, and R.~Devore}, {\em {Adaptive finite element methods
  with convergence rates}}, Numer. Math., 97 (2004), pp.~219--268.

\bibitem{FDBolintineanu}
{\sc D.~S. Bolintineanu, A.~Sayyed-Ahmad, H.~T. Davis, and Y.~N. Kaznessis},
  {\em {Poisson-Nernst-Planck models of nonequilibrium ion electrodiffusion
  through a protegrin transmembrane pore}}, PLOS Comput. Biol., 5 (2009),
  p.~e1000277.

\bibitem{Brenner}
{\sc S.~C. Brenner and L.~R. Scott}, {\em {The Mathematical Theory of Finite
  Element Methods}}, Springer-Verlag, 1998.

\bibitem{Brezzi}
{\sc F.~Brezzi, A.~C.~S. Capelo, and L.~Gastaldi}, {\em A singular perturbation
  analysis of reverse-biased semiconductor diodes}, SIAM J. Math. Anal., 20
  (1989), pp.~372--387.

\bibitem{Brezzi2005}
{\sc F.~Brezzi, L.~D. Marini, S.~Micheletti, P.~Pietra, R.~Sacco, and S.~Wang},
  {\em {Discretization} of {semiconductor device problems}}, Handb. Numer.
  Anal., 13 (2005), pp.~317--441.

\bibitem{babuska}
{\sc I.~Bubuka and M.~Vogelius}, {\em Feedback and adaptive finite element
  solution of one-dimensional boundary value problems}, Numer. Math., 44
  (1984), pp.~75--102.

\bibitem{FDCardenas}
{\sc A.~E. C{\'a}rdenas, R.~D. Coalson, and M.~G. Kurnikova}, {\em
  {Three-dimensional Poisson-Nernst-Planck theory studies: influence of
  membrane electrostatics on gramicidin a channel conductance}}, Biophys. J.,
  79 (2000), pp.~80--93.

\bibitem{Carstensen}
{\sc C.~Carstensen and G.~Dolzmann}, {\em A posteriori error estimates for
  mixed {finite element method} in elasticity}, Numer. Math., 81 (1998),
  pp.~187--209.

\bibitem{cheddadi2009guaranteed}
{\sc I.~Cheddadi, R.~Fu{\v{c}}{\'{i}}k, M.~I. Prieto, and M.~Vohralik}, {\em Guaranteed and robust a posteriori error estimates for singularly perturbed reaction-diffusion problems}, ESAIM-Math. Model. Num., 43 (2009), pp.~867--888.

\bibitem{PBchenlong}
{\sc L.~Chen, M.~Holst, and J.~Xu}, {\em {The finite element approximation of the nonlinear Poisson-Boltzmann equation}}, SIAM J. Numer. Anal., 45 (2007),
  pp.~2298--2320.

\bibitem{CIARLET199117}
{\sc P.~G. Ciarlet}, {\em Basic error estimates for elliptic problems}, Handb.
  Numer. Anal., 2 (1991), pp.~17--351.

\bibitem{demlow2016maximum-norm}
{\sc A.~Demlow and N.~Kopteva}, {\em Maximum-norm a posteriori error estimates
  for singularly perturbed elliptic reaction-diffusion problems}, Numer. Math.,
  133 (2016), pp.~707--742.

\bibitem{Ding19}
{\sc J.~Ding, Z.~Wang, and S.~Zhou}, {\em Positivity preserving finite
  difference methods for {Poisson-Nernst-Planck} equations with steric
  interactions: Application to slit-shaped nanopore conductance}, J. Comput.
  Phys.,  (2019), p.~108864.

\bibitem{electrodiffusion1}
{\sc I.~Dione, N.~Doyon, and J.~Deteix}, {\em Sensitivity analysis of the
  {Poisson-Nernst-Planck} equations: a finite element approximation for the
  sensitive analysis of an electrodiffusion model}, J. Math. Biol.,  (2018),
  pp.~1--36.

\bibitem{Dorfler1996A}
{\sc W.~D{\"o}rfler}, {\em A convergent adaptive algorithm for {Poisson's}
  equation}, SIAM J. Numer. Anal., 33 (1996), pp.~1106--1124.

\bibitem{Eriksson}
{\sc K.~Eriksson and C.~Johnson}, {\em {Error estimates and automatic time step
  control for nonlinear parabolic problems, I}}, SIAM J. Numer. Anal., 24
  (1987), pp.~12--23.

\bibitem{Gajewski}
{\sc H.~Gajewski}, {\em On Uniqueness and Stability of Steady-State Carrier
  Distributions in Semiconductors}, Springer Berlin Heidelberg, 1986.

\bibitem{SGajewski}
{\sc H.~Gajewski and K.~Groger}, {\em On the basic equations for carrier
  transport in semiconductors}, J. Math. Anal. Appl., 113 (1986), pp.~12--35.

\bibitem{FEsungao}
{\sc H.~Gao and P.~Sun}, {\em {A linearized local conservative mixed finite
  element method for Poisson-Nernst-Planck equations}}, J. Sci. Comput., 77
  (2018), pp.~793--817.

\bibitem{Gilbarg}
{\sc T.~N.~S. Gilbarg~D}, {\em { Elliptic Partial Differential Equations of
  Second Order}}, Springer-Verlag, 1983.

\bibitem{AGolovnev}
{\sc A.~Golovnev and S.~Trimper}, {\em Steady state solution of the
  {Poisson-Nernst-Planck} equations}, Phys. Lett. A, 374 (2010),
  pp.~2886--2889.

\bibitem{golovnev2011JCP}
{\sc A.~Golovnev and S.~Trimper}, {\em Analytical solution of the
  {Poisson-Nernst-Planck} equations in the linear regime at an applied
  dc-voltage}, J. Chem. Phys., 134 (2011), p.~154902.

\bibitem{Hayeck1990EXISTENCE}
{\sc N.~Hayeck, A.~Nachaoui, and N.~R. Nassif}, {\em Existence and regularity
  for {Van Roosbroeck} systems with general mixed boundary conditions}, COMPEL,
  9 (1990), pp.~217--228.

\bibitem{FEsunhem}
{\sc M.~He and P.~Sun}, {\em Mixed finite element analysis for the
  {Poisson-Nernst-Planck/Stokes} coupling}, J. Comput. Appl. Math., 341 (2018),
  pp.~61--79.

\bibitem{SPHollerbach}
{\sc U.~Hollerbach, D.~P. Chen, and R.~S. Eisenberg}, {\em {Two- and
  three-dimensional Poisson-Nernst-Planck simulations of current flow through
  gramicidin A}}, J. Sci. Comput., 16 (2001), pp.~373--409.

\bibitem{holst2000jcp}
{\sc M.~Holst, N.~A. Baker, and F.~Wang}, {\em {Adaptive multilevel finite
  element solution of the Poisson-Boltzmann equation I: Algorithms and
  examples}}, J. Comput. Chem., 21 (2000), pp.~1319--1342.

\bibitem{horng2012JPCB}
{\sc T.~Horng, T.~Lin, C.~Liu, and B.~Eisenberg}, {\em {PNP} equations with
  steric effects: {A} model of ion flow through channels}, J. Phys. Chem. B,
  116 (2012), pp.~11422--11441.

\bibitem{hu2020fully}
{\sc J.~Hu and X.~Huang}, {\em A fully discrete positivity-preserving and
  energy-dissipative finite difference scheme for {Poisson--Nernst--Planck}
  equations}, Numer. Math.,  (2020), pp.~1--39.

\bibitem{electrodiffusion2}
{\sc J.~J. Jasielec, R.~Filipek, K.~Szyszkiewicz, J.~Fausek, M.~Danielewski,
  and A.~Lewenstam}, {\em Computer simulations of electrodiffusion problems
  based on {Nernst-Planck and Poisson} equations}, Comp. Mater. Sci., 63
  (2012), pp.~75--90.

\bibitem{Jerome1}
{\sc J.~W. Jerome}, {\em Consistency of {semiconductor modeling}: {An
  existence/stability analysis} for the {stationary Van Roosbroeck system}},
  SIAM J. Appl. Math., 45 (1985), pp.~565--590.

\bibitem{FEjerome}
{\sc J.~W. Jerome and T.~Kerkhoven}, {\em {A finite element approximation
  theory for the drift diffusion semiconductor model}}, SIAM J. Numer. Anal.,
  28 (1991), pp.~403--422.

\bibitem{Jiang2014JPCM}
{\sc J.~Jiang, D.~Cao, D.-e. Jiang, and J.~Wu}, {\em Time-dependent density
  functional theory for ion diffusion in electrochemical systems}, J. Phys.
  Condens. Mat., 26 (2014), p.~284102.

\bibitem{kilic2007bPRE}
{\sc M.~S. Kilic, M.~Z. Bazant, and A.~Ajdari}, {\em Steric effects in the
  dynamics of electrolytes at large applied voltages. {II}. modified
  {Poisson-Nernst-Planck} equations}, Phys. Rev. E, 75 (2007), p.~021503.

\bibitem{li2018analysis}
{\sc Y.~Li}, {\em Analysis of novel adaptive two-grid finite element algorithms
  for linear and nonlinear problems}, preprint arXiv:1805.07887,  (2018).

\bibitem{liu2017free}
{\sc H.~Liu and Z.~Wang}, {\em A free energy satisfying discontinuous galerkin
  method for one-dimensional {Poisson-Nernst-Planck} systems}, J. Comput.
  Phys., 328 (2017), pp.~413--437.

\bibitem{SPLiu}
{\sc W.~Liu}, {\em {Geometric singular perturbation approach to steady-state
  Poisson-Nernst-Planck systems}}, SIAM J. Appl. Math., 65 (2005),
  pp.~754--766.

\bibitem{FELu}
{\sc B.~Lu, M.~Holst, J.~A. Mccammon, and Y.~Zhou}, {\em
  {Poisson-Nernst-Planck} equations for simulating biomolecular
  diffusion-reaction processes {I}: {Finite} element solutions}, J. Comput.
  Phys., 229 (2010), pp.~6979--6994.

\bibitem{lu2011BJ}
{\sc B.~Lu and Y.~Zhou}, {\em {Poisson-Nernst-Planck} equations for simulating
  biomolecular diffusion-reaction processes {II}: Size effects on ionic
  distributions and diffusion-reaction rates}, Biophys. J., 100 (2011),
  pp.~2475--2485.

\bibitem{electrodiffusion3}
{\sc B.~Lu, Y.~Zhou, G.~A. Huber, S.~D. Bond, M.~J. Holst, and J.~A. Mccammon},
  {\em Electrodiffusion: a continuum modeling framework for biomolecular
  systems with realistic spatiotemporal resolution}, J. Chem. Phys., 127
  (2007), p.~135102.

\bibitem{FVMathur}
{\sc S.~R. Mathur and J.~Y. Murthy}, {\em {A multigrid method for the
  Poisson-Nernst-Planck equations}}, Int. J. Heat. Mass. Tran., 52 (2009),
  pp.~4031--4039.

\bibitem{SMauri}
{\sc A.~Mauri, A.~Bortolossi, G.~Novielli, and R.~Sacco}, {\em {3D} finite
  element modeling and simulation of industrial semiconductor devices including
  impact ionization}, J. Math. Ind., 5 (2015), p.~1.

\bibitem{metti2016energetically}
{\sc M.~S. Metti, J.~Xu, and C.~Liu}, {\em Energetically stable discretizations
  for charge transport and electrokinetic models}, J. Comput. Phys., 306
  (2016), pp.~1--18.

\bibitem{mock2}
{\sc M.~S. Mock}, {\em {Analysis of Mathematical Models of Semiconductor
  Devices}}, Boole Press, Dublin, 1983.

\bibitem{Morin2001Data}
{\sc P.~Morin, R.~H. Nochetto, and K.~G. Siebert}, {\em {Data oscillation and
  convergence of adaptive FEM}}, SIAM J. Numer. Anal., 38 (2001), pp.~466--488.

\bibitem{schonke2012JPA}
{\sc J.~Schonke}, {\em Unsteady analytical solutions to the
  {Poisson-Nernst-Planck} equations}, J. Phys. A: Math. Theor., 45 (2012),
  p.~455204.

\bibitem{Scootzhang}
{\sc L.~R. Scott and S.~Zhang}, {\em Finite element interpolation of nonsmooth
  functions satisfying boundary conditions}, Math. Comp., 54 (1990),
  pp.~483--493.

\bibitem{singer2008EJAM}
{\sc A.~Singer, D.~Gillespie, J.~Norbury, and R.~S. Eisenberg}, {\em Singular
  perturbation analysis of the steady-state {Poisson-Nernst-Planck} system:
  {Applications} to ion channels}, Eur. J. Appl. Math., 19 (2008),
  pp.~541--560.

\bibitem{FEsongzhang}
{\sc Y.~Song, Y.~Zhang, C.~L. Bajaj, and N.~A. Baker}, {\em {Continuum
  diffusion reaction rate calculations of wild-type and mutant mouse
  acetylcholinesterase: Adaptive finite element analysis}}, Biophys. J., 87
  (2004), pp.~1558--1566.

\bibitem{FEsongshen}
{\sc Y.~Song, Y.~Zhang, T.~Shen, C.~L. Bajaj, J.~A. Mccammon, and N.~A. Baker},
  {\em {Finite element solution of the steady-state Smoluchowski equation for
  rate constant calculations}}, Biophys. J., 86 (2004), pp.~2017--2029.

\bibitem{tu2015CPC}
{\sc B.~Tu, Y.~Xie, L.~Zhang, and B.~Lu}, {\em Stabilized finite element
  methods to simulate the conductances of ion channels}, Comput. Phys. Commun.,
  188 (2015), pp.~131--139.

\bibitem{VerfurthR}
{\sc R.~Verf{\"u}rth}, {\em A posteriori error estimators for the {Stokes}
  equations}, Numer. Math., 55 (1989), pp.~309--325.

\bibitem{Verfuhrt}
{\sc R.~Verf{\"u}rth}, {\em A posteriori error estimates for nonlinear
  problems. {Finite} element discretizations of elliptic equations}, Math.
  Comp., 62 (1994), pp.~445--475.

\bibitem{Verfuhrtgeneral}
{\sc R.~Verf{\"{u}}rth}, {\em A posteriori error estimation and adaptive
  mesh-refinement techniques}, J. Comput. Appl. Math., 50 (1996), pp.~67--83.

\bibitem{verfurth1998robust}
{\sc R.~Verf{\"u}rth}, {\em Robust a posteriori error estimators for a
  singularly perturbed reaction-diffusion equation}, Numer. Math., 78 (1998),
  pp.~479--493.

\bibitem{wei2012SIAM}
{\sc G.~W. Wei, Q.~Zheng, Z.~Chen, and K.~Xia}, {\em Variational multiscale
  models for charge transport}, SIAM Rev., 54 (2012), pp.~699--754.

\bibitem{FVSrinivasan}
{\sc J.~Wu, V.~Srinivasan, J.~Xu, and C.~Wang}, {\em {Newton-Krylov-Multigrid
  algorithms for battery simulation}}, J. Electrochem. Soc., 149 (2002),
  pp.~A1342--A1348.

\bibitem{xie2013MBMB}
{\sc Y.~Xie, J.~Cheng, B.~Lu, and L.~Zhang}, {\em {Parallel adaptive finite
  element algorithms for solving the coupled electro-diffusion equations}},
  Mol. Based Math. Biol., 1 (2013), pp.~90--108.

\bibitem{xu2014MBMB}
{\sc S.~Xu, M.~Chen, S.~Majd, X.~Yue, and C.~Liu}, {\em {Modeling and
  simulating asymmetrical conductance changes in gramicidin pores}}, Mol. Based
  Math. Biol., 2 (2014), pp.~34--55.

\bibitem{xu2014PRE}
{\sc Z.~Xu, M.~Ma, and P.~Liu}, {\em Self-energy-modified
  {Poisson-Nernst-Planck} equations: {WKB} approximation and finite-difference
  approaches}, Phys. Rev. E, 90 (2014), p.~013307.

\bibitem{yang}
{\sc Y.~Yang, B.~Lu, and Y.~Xie}, {\em {A decoupling two-grid method for the
  steady-state Poisson-Nernst-Planck equations}}, preprint arXiv:1609.02277,
  (2016).

\bibitem{Zeidler}
{\sc E.~Zeidler}, {\em Nonlinear Functional Analysis and its Applications},
  Springer-Verlag, Berlin, 1988.

\bibitem{zhang2015guaranteed}
{\sc B.~Zhang, S.~Chen, and J.~Zhao}, {\em Guaranteed a posteriori error
  estimates for nonconforming finite element approximations to a singularly
  perturbed reaction-diffusion problem}, Appl. Numer. Math., 94 (2015),
  pp.~1--15.

\bibitem{FDZheng}
{\sc Q.~Zheng, D.~Chen, and G.~W. Wei}, {\em {Second-order
  Poisson-Nernst-Planck solver for ion channel transport}}, J. Comput. Phys.,
  230 (2011), pp.~5239--5262.

\bibitem{BFZhou}
{\sc Y.~Zhou, B.~Lu, G.~A. Huber, M.~J. Holst, and J.~A. McCammon}, {\em
  Continuum simulations of acetylcholine consumption by acetylcholinesterase:
  {A Poisson-Nernst-Planck} approach}, J. Phys. Chem. B, 112 (2008),
  pp.~270--275.





\end{thebibliography}


\end{document}